\documentclass[11pt,reqno, a4paper]{amsart}

\usepackage{amssymb}

\usepackage[english]{babel}
\usepackage{mathrsfs}

\usepackage[dvipsnames,svgnames,x11names]{xcolor}
\usepackage[hyphens]{url}
\usepackage[colorlinks=true,linkcolor=Maroon,citecolor=blue,urlcolor=blue,hypertexnames=false,linktocpage]{hyperref}
\usepackage{bookmark}
\usepackage{amsmath,thmtools}
\usepackage{bm}
\usepackage{mathtools}
\mathtoolsset{showonlyrefs} 

\usepackage{fancyhdr}
\usepackage{esint}
\usepackage{enumerate}
\usepackage{enumitem} 

\usepackage[scr]{rsfso}
\usepackage{mathtools}
\usepackage{pictexwd,dcpic}
\usepackage{graphicx}

\newcommand{\Snp}{\mathbb S^{n+1}_+}
\newcommand{\ph}{\phi}
\newcommand{\dd}{\,d}

\def\avint{\mathop{\mathchoice{\,\rlap{-}\!\!\int}
		{\rlap{\raise.15em{\scriptstyle -}}\kern-.2em\int}
		{\rlap{\raise.09em{\scriptscriptstyle -}}\!\int}
		{\rlap{-}\!\int}}\nolimits}

\usepackage[labelfont=bf, font=small]{caption}

\newcounter{mgncount}

\declaretheorem[name=Theorem,numberwithin=section]{thm}

\declaretheorem[name=Lemma,sibling=thm]{lemma}
\declaretheorem[name=Proposition,sibling=thm]{prop}

\declaretheorem[name=Corollary,sibling=thm]{cor}
\declaretheorem[name=Assumption,style=definition,sibling=thm]{assum}

\numberwithin{equation}{section}

\newcommand{\bbR}{\mathbb{R}}

\newcommand{\bbS}{\mathbb{S}}

\newcommand{\ka}{\kappa}

\newcommand{\om}{\omega_n}

\newcommand{\Om}{\Omega}

\newcommand{\Ga}{\Gamma}

\usepackage{dsfont}

\newcommand{\onabla}{\overline{\nabla}} 

\newcommand{\pt}{ \partial_t} 
\newcommand{\dotF}{\dot{F}}

\newcommand{\avg}[1]{\langle #1 \rangle}
\def\restrict#1{\raise-.5ex\hbox{\ensuremath|}_{#1}}

\usepackage{scalerel}[2014/03/10]
\usepackage[usestackEOL]{stackengine}
\def\intavg{\,\ThisStyle{\ensurestackMath{%
			\stackinset{c}{0\LMpt}{c}{0\LMpt}{\SavedStyle-}{\SavedStyle\phantom{\int}}}%
		\setbox0=\hbox{$\SavedStyle\int\,$}\kern-\wd0}\int}

\usepackage{graphicx}

\numberwithin{equation}{section}

\usepackage{xcolor}
\definecolor{color-cites}{HTML}{9a2144}
\definecolor{brickred}{HTML}{e63d12}
\definecolor{royalblue}{HTML}{2054e3}
\definecolor{blavet}{HTML}{3adbcc}
\definecolor{verd}{HTML}{4a6741}

\usepackage{tikz}

\usepackage{cancel}
\usepackage{tcolorbox}

\begin{document}
	
	\title[QP-CF with horo-convexity on the sphere]{The quermassintegral preserving curvature flow \smallskip \\ for horo-convex hypersurfaces  in the sphere}

	\begin{abstract}
		We introduce fully nonlinear curvature flows of horo-convex hypersurfaces in the sphere that preserve an arbitrarily prescribed spherical quermassintegral. The non-local normalization is defined relative to a fixed ambient origin, and we prove that the evolution is governed by a single smooth fixed-origin equation for all time. 
		
		For monotone, homogeneous curvature functions satisfying concavity and inverse concavity, we establish preservation of horo-convexity, uniform curvature pinching, a direct estimate for the non-local coefficient. A Tso-type argument then yields global curvature bounds and long-time existence. We further prove exponential decay of the traceless second fundamental form and exponential $C^\infty$ convergence to the geodesic sphere centered at the fixed origin whose radius is determined by the preserved quermassintegral.
	\end{abstract}

	\author{Sara Albert-Niclòs}
	\address{\flushleft\parbox{\linewidth}{{\bf Sara Albert-Niclòs} \\Universitat de Val\`encia\\ Departament de Matem\`atiques\\ Av. Vicent Andrés Estellés~19\\ 46100 Burjassot\\ Spain\\ {\href{mailto:sara.albert@uv.es}{sara.albert@uv.es}}}}

	\author{Esther Cabezas-Rivas}
	\address{\flushleft\parbox{\linewidth}{{\bf Esther Cabezas-Rivas} \\Universitat de Val\`encia\\ Departament de Matem\`atiques\\ Av. Vicent Andrés Estellés~19\\ 46100 Burjassot\\ Spain\\ {\href{mailto:esther.cabezas-rivas@uv.es}{esther.cabezas-rivas@uv.es}}}}
	
	\author{Shujing Pan}
	\address{\flushleft\parbox{\linewidth}{{\bf Shujing Pan} \\Goethe-Universit\"at\\ Institut f\"ur Mathematik\\ Robert-Mayer-Str.~10\\ 60325 Frankfurt\\ Germany\\ {\href{mailto:pan@math.uni-frankfurt.de}{pan@math.uni-frankfurt.de}}}}

	\date{\today. This work was funded by the German Research Foundation (Deutsche Forschungsgemeinschaft,
		DFG) through the project “Curvature flows with local and nonlocal sources”, SCHE 1879/4-1.
		The first and second authors have been partially supported by project PID2022-136589NB-I00, funded by MCIN/AEI/10.13039/501100011033 and by FEDER, A way of making Europe. The first author has also been supported by the predoctoral grant PREP2022-000011 associated with this project. The first and second authors have also been partially supported by project CIAICO/2023/035, funded by the Conselleria d’Educació, Cultura, Universitats i Ocupació. 
	}

	\maketitle

	\section{Introduction and main results}
	
	Constrained curvature flows provide a geometric mechanism for
	deforming hypersurfaces towards canonical shapes while
	preserving a prescribed global quantity. Their analysis requires a
	delicate balance between the local curvature speed and a non-local
	normalization, and this interaction becomes particularly subtle in
	positively curved ambient spaces.

	The classical model is Huisken's volume-preserving mean-curvature flow  in
	Euclidean space \cite{Huisken_1987_VPMCF}. Starting from a strictly convex hypersurface, the
	flow exists smoothly for all time, preserves the enclosed volume,
	decreases the surface area and converges to a round sphere. In particular, it gives
	a dynamical proof of the Euclidean isoperime\-tric inequality.  Huisken's theorem illustrates how preservation of convexity, global existence and convergence to a canonical shape can combine to yield geometric inequalities.
	
	In the sphere, however, the direct analogue of the volume-preserving mean curvature flow does not in general preserve convexity. As observed by Huisken, nearly totally geodesic portions lying close to the equator may be pushed across it by the non-local forcing term and thereby lose convexity. This obstruction is a genuinely spherical effect; indeed, in hyperbolic space it was shown by the second author and Miquel \cite{CaMi1} that the effects of the ambient curvature can be handled by replacing ordinary convexity with the more natural notion of horo-convexity.
	
	The second author and Scheuer addressed this spherical difficulty by introducing an origin-dependent radial weight in the non-local forcing \cite{CabezasRivasScheuer2024}. This factor vanishes at the equator and slows down the evolution precisely where convexity may be lost. Such a degeneracy nevertheless comes at a price: the same  weight occurs in the denominator defining the global coefficient, which may therefore become unbounded as the hypersurface approaches the equator.
	
	To overcome this, in \cite{CabezasRivasScheuer2024} the global coefficient is controlled by selecting suitable origins on successive time intervals. This produces quermassintegral-preserving evolutions with global existence and convergence to geodesic spheres, but requires finitely many recalibrations of the origin. Consequently, the evolution is smooth only between consecutive recalibrations and merely Lipschitz in time when the origin is changed. As a by-product, a monotonicity identity established for one origin need not persist across the complete evolution, restricting the use of the flow in deriving geometric inequalities.
	
The key idea of the present paper is to combine the radial mechanism of \cite{CabezasRivasScheuer2024} with the spherical horo-convexity introduced by the third author and Scheuer in \cite{PanScheuer2025}, as it is the direct counterpart of the notion of convexity that was successful in the hyperbolic setting, but adapted to the spherical geometry. This combination allows us to control the no-local term relative to one fixed origin and eliminates the need for recalibrations. We further replace the mean curvature by a broad class of fully nonlinear curvature functions, and obtain a single smooth quermassintegral-preserving evolution  on the whole
time interval. Besides giving global existence and exponential convergence, this provides a suitable framework for applications based on global monotonicity formulas and geometric inequalities.

	 Let $n \geq 2$, and let
	\(
	x_0:\mathbb S^n\longrightarrow\mathbb S^{n+1}_+
	\)
	be a smooth embedding whose image
	\(
	M_0:=x_0(\mathbb S^n) \subset\mathbb S^{n+1}_+
	\)
	is a smooth  and closed
	hypersurface contained in the open northern hemisphere $\mathbb S^{n+1}_+$ of the unit
	sphere with the round metric $\bar g$. We fix the north pole as an origin $\mathcal O$ and regard $\mathbb{S}^{n+1}_+$ as a warped product
	\begin{equation} \label{warped}
	\mathbb S^{n+1}_+
	=
	\left(0,\frac{\pi}{2}\right)\times\mathbb S^n,
	\qquad
	\bar g=dr^2+\phi^2(r)\sigma,
	\qquad
	\phi(r)=\sin r,
	\end{equation}
	where $r$ is the radial distance from $\mathcal O$ and $\sigma$ is the
	standard round metric on $\mathbb S^n$. We denote by $\Omega_0$ the domain enclosed by $M_0$.

	We consider a family of embeddings $x = x(\cdot,t)$ evolving by 
	\begin{equation} 
	\partial_t x = \left(\mu(t)\phi' - F\right)\nu,
	\label{eq:flow}
	\end{equation}
	with initial condition $x(\cdot,0)=x_0$. Here $\nu$ is the outward unit normal
	to $M_t:=x(\mathbb S^n,t)$, and $F$ is a symmetric function of the principal
	curvatures. With our sign convention, $-F\nu$ points inwards on strictly
	convex hypersurfaces (i.e. with $\kappa_1 \geq 0$, where $\kappa_1 \leq \cdots \leq \kappa_n$ denote the principal curvatures).
	
	For a fixed $\ell\in\{0,\ldots,n\}$, the global coefficient is defined
	by
	\begin{equation} \label{def-mu}
	\mu(t)
	:=
	\frac{
		\displaystyle\int_{M_t}F\sigma_\ell\,dV_t
	}{
		\displaystyle\int_{M_t}\phi'\sigma_\ell\,dV_t
	},
	\end{equation}
	where $\sigma_\ell$ denotes the $\ell$-th elementary symmetric
	polynomial of
	degree $\ell$ in the principal curvatures, and $d V_t$ represents the volume element of $M_t$.  By the first variation
	formula for the spherical quermassintegrals, this choice gives
	\(
	W_\ell(\Omega_t)=W_\ell(\Omega_0),
	\)
	where $\Omega_t$ is the domain enclosed by $M_t$. Thus \eqref{eq:flow} preserves any prescribed spherical quermassintegral. 
	
	As argued above, the natural geometric class for this problem is given by spherical
	horo-convex (or $h$-convex) hypersurfaces. In
	the present notation, horo-convexity is the tensor inequality
	\begin{equation}
	\phi' h\geq(1-u)g,
	\qquad
	u:=\langle\phi \, \partial_r,\nu\rangle,
	\label{def:horoconvex}
	\end{equation}
	where $h$ and $g$ are the second fundamental form and the induced metric, respectively,  and $u$ stands for the support function with respect to the fixed origin $\mathcal O$.

	We impose the following structural assumptions on the curvature
	function.
	\begin{assum} \label{ass:F}
		Let 
		\[{\Ga_{+} = \{\ka\in \bbR^{n} \, | \, \ka_{i}>0, \text{ for all } i=1, \ldots,n\},}\]
		be the positive cone, and consider a symmetric function 
		$F \in C^\infty(\Gamma_+)$. We assume that 
		\begin{enumerate}[label=(\roman*)]
			\item[(a)]  $F$ is strictly increasing in each argument;
			\item[(b)] \label{assum:F_homog_deg_1} $F$ is homogeneous of degree one;
			\item[(c)] $F$ is normalized so that $F(1,\dots,1)=n$;
			\item[(d)] $F$ is concave;
			\item[(e)] $F$ is inverse concave, i.e. the function $\kappa = (\kappa_i) \mapsto \frac{1}{F(\kappa_1^{-1},\dots,\kappa_n^{-1})}$
			is concave on $\Gamma_+$.
		\end{enumerate}
	\end{assum}

	Now we are in a position to state our main result:
	
	\begin{thm}\label{thm:main}
		Let $n\geq 2$, let $F$ satisfy Assumption~\ref{ass:F}, and $x_0: \mathbb S^n \rightarrow \bbS^{n+1}_+$  be a smooth embedding such that
		\(
		M_0:=x_0(\mathbb S^n)
		\)
		is strictly horo-convex and bounds a domain $\Omega_0$. Then, for every $\ell\in\{0,\ldots,n\}$, 
		\begin{enumerate}
		\item[{\rm (a)}] the problem \eqref{eq:flow}
		with $\mu$  as in \eqref{def-mu} admits a unique smooth
		solution
		\[
		x:\mathbb S^n\times[0,\infty)
		\longrightarrow
		\mathbb S^{n+1}_+,
		\] 
		with $x(\cdot,0) = x_0$.
		\item[{\rm (b)}] for every $t\geq 0$, the evolving hypersurfaces $M_t$ remain strictly horo-convex, and 
		\item[{\rm (c)}] as $t \to \infty$ $M_t$ converges exponentially fast in the $C^\infty$-topology to the
		geodesic sphere $S_{r_\infty}(\mathcal O)$ centered at the fixed origin,
		where the radius $r_\infty\in(0,\pi/2)$ is uniquely determined by
		\(
		W_\ell(B_{r_\infty}(\mathcal O))
		=
		W_\ell(\Omega_0).
		\)
		\end{enumerate}
	\end{thm}

	Theorem \ref{thm:main} shows that combining the radial correction with spherical horo-convexity provides the geometric structure needed to control the non-local coefficient relative to a single ambient origin. This mechanism distinguishes the present construction from a mere fully nonlinear extension of the  flow in \cite{CabezasRivasScheuer2024}: it removes all recalibrations and allows the same radial weight and the same evolution equation to be used throughout the entire process.

More specifically, Theorem~\ref{thm:main} extends \cite{CabezasRivasScheuer2024} in three directions. First, the mean curvature is replaced by a broad class of fully nonlinear curvature functions satisfying natural concavity and inverse-concavity conditions. Second, we derive a direct estimate for the non-local coefficient that produces a single smooth fixed-origin evolution on $[0,\infty)$. In addi\-tion, we establish exponential $C^\infty$-convergence to the final centered sphere, whereas the argument in \cite{CabezasRivasScheuer2024} does not provide an exponential rate towards that fixed limiting sphere.

The fixed-origin formulation also provides a framework for applications to geometric inequalities, since origin-dependent functionals admit evolution formulas along the entire flow and may therefore be used to investigate global weighted monotonicity formulas.  To the best of our knowledge, \eqref{eq:flow} is the first fully nonlinear non-local curvature flow in the sphere to combine preservation of an arbitrarily prescribed quermassintegral, a single smooth fixed-origin evolution for all time, and exponential convergence to the corresponding centered geodesic sphere.

	We briefly describe the main ingredients of the proof. The first 	difficulty is the preservation of spherical horo-convexity, which we encode as a tensorial inequality (see Section \ref{sec:S-tensor}). 
	For a fully nonlinear speed, the evolution of the relevant quantity contains terms involving
	the second derivatives of $F$ that must be combined with the correction
	arising from the smallest-eigenvalue formula at a hypothetical
	first null eigenvector. The inverse-concavity
	assumption provides the quadratic inequality needed to control this
	combination and preserve horo-convexity. A second tensor maximum-principle argument yields the
uniform curvature pinching estimate
	$\kappa_n\leq C_0\kappa_1$.
	
	The next issue is the control of the non-local coefficient. The
	pinching estimate, together with the preserved quermassintegral,
	gives uniform inner and outer radius bounds. We then establish
	weighted curvature-integral estimates and deduce
	\(
	\mu(t)
	\leq
	C\max_{M_t}F.
	\)
	This estimate is the central mechanism that allows the origin to
	remain fixed. In
	particular, it makes a
	Tso-type curvature argument  possible: at a spatial maximum of the
	auxiliary function, the positive contributions involving the
	non-local coefficient are at most quadratic, whereas the leading negative
	curvature term is cubic. This yields uniform bounds for the curvature
	and for $\mu$.
	
	These estimates prevent the formation of finite-time singularities and imply long-time existence.
	We then improve the curvature pinching to obtain exponential decay of
	the traceless second fundamental form. Quantitative
	almost-umbilical rigidity shows that the evolving hypersurfaces
	approach the family of geodesic spheres exponentially, while the
	preserved quermassintegral determines the limiting radius. This asymptotic spherical closeness, in turn,  yields uniform regularity of all orders.
	
	It remains to identify the center and obtain the final rate of
	convergence. Since the fixed origin is not assumed to lie inside the
	enclosed domains, the global radial parametrization used in
	\cite[Section~7]{CabezasRivasScheuer2024} is not available. We
	instead adapt the argument at a maximum point of the radial distance,
	which suffices to exclude every noncentered geodesic sphere as a
	subsequential limit. Once smooth convergence to the centered sphere
	has been established, we represent the hypersurfaces as radial graphs
	over this fixed limiting sphere. A maximum-principle argument
	gives exponential decay of the graph functions in $C^0$, and the
	uniform higher-order estimates and interpolation yield exponential
	convergence in $C^\infty$.

	The paper is organized as follows. In Section~\ref{sec:preliminaries}, we introduce  notational conventions, recall the spherical quermassintegrals and
	describe the geometric interpretation of 
	horo-convexity. Section~\ref{sec:evol} contains the identities and evolution
	equations used throughout the paper. In Section~\ref{sec:preservation-hconvexity}, we prove
	preservation of horo-convexity and establish the uniform curvature
	pinching estimate. Section~\ref{sec:geometric-estimates} is devoted to the radius estimates, the
	weighted curvature-integral bounds and the control of the non-local
	coefficient. In Section~\ref{sec:upper-curvature}, we obtain the uniform upper curvature
	estimate by a Tso-type argument. Finally, Section~\ref{sec:long-time-convergence} proves long-time
	existence, exponential almost-umbilicity, quantitative closeness to
	geodesic spheres and exponential $C^\infty$-convergence to the
	geodesic sphere centered at $\mathcal O$.

	\section{Notation, conventions and preliminary results}
	\label{sec:preliminaries}
	
	\subsection{Geometric conventions}

	Let
	\(
	x:M^n\longrightarrow\mathbb S^{n+1}_+
	\)
	be a smooth, closed and embedded hypersurface bounding a domain
	$\Omega$. We denote by $\nu$ its outward unit normal, by
	\(
	g=x^*\bar g
	\)
	the induced metric and by $\nabla$ its Levi-Civita connection. The
	ambient connection and curvature tensor are denoted by
	$\onabla$ and $\overline R$, respectively.
	
	Our convention for the second fundamental form $h$ is fixed by the Gauss formula
	\begin{align}
	\onabla_X Y = \nabla_X Y - h(X,Y)\nu, 
	\label{eq:Gauss_formula}
	\end{align}
	for vector fields $X,Y$ tangent to the hypersurface. Throughout this article, we adopt the sign curvature convention so that for $\mathbb S^{n+1}$ it holds
	\begin{equation} \label{ambient-R}
	\bar R_{\alpha\beta\gamma\delta}
	=
	\bar g_{\alpha\gamma}\bar g_{\beta\delta}
	-
	\bar g_{\alpha\delta}\bar g_{\beta\gamma}.
	\end{equation}

	We recall that in the warped product model \eqref{warped} the radial vector field $\phi\,\partial_r$ is a conformal Killing field satisfying
	\begin{equation}
	\onabla_X(\phi \,\partial_r)=\phi'\,X
	\label{eq:conformalfield}
	\end{equation}
	for every ambient vector field $X$ (see \cite[Lemma 2.1]{GuanLi2015}).

	We write
	\[
	h_i^{j}:=g^{jk}h_{ik},
	\qquad
	(h^2)_i^{j}:=h_i^{k}h_k^{j},
	\]
	and denote the Weingarten operator by
	\(
	A=(h_i^{j}).
	\)
	Its eigenvalues are the principal curvatures, ordered as
	\(
	\kappa_1\leq\cdots\leq\kappa_n.
	\)
	The mean curvature is
	\(
	H:=\operatorname{tr}_gA
	=\sum_{i=1}^n\kappa_i.
	\)
	Finally,  the traceless second fundamental form is
	\(
	\mathring h_{ij}
	:=
	h_{ij}-\frac{H}{n}g_{ij}.
	\)
	
	Hereafter, indices are raised and lowered with the induced metric $g$ and we use the Einstein sum convention that repeated
	indices above and below are summed from $1$ to $n$. Unless otherwise stated, all
	tensor norms and covariant derivatives are computed with respect to
	$g$.

	\subsection{Geometric interpretation of horo-convexity} 
	The novel notion of spherical horo-convexity was introduced in
	\cite{PanScheuer2025}. Its original motivation comes from the
	geometric characterization proved in
	\cite[Proposition~2.1]{PanScheuer2025}. Namely, geodesic spheres
	contained in the northern hemisphere
	$\mathbb S^{n+1}_+$ and tangent to the equator
	$\partial\mathbb S^{n+1}_+$ satisfy equality in \eqref{def:horoconvex}.
	
	These equator-tangent geodesic spheres are regarded as the spherical
	counterparts of hyperbolic horospheres. Accordingly, \eqref{def:horoconvex}
	can be interpreted as a one-sided curvature condition with respect to
	this distinguished family of model hypersurfaces. A hypersurface
	satisfying \eqref{def:horoconvex} is called horo-convex, or
	$h$-convex for short.
	
	The terminology also admits a natural conformal interpretation.
	Indeed, the spherical horo-convexity condition arises naturally by
	transporting the usual hyperbolic notion through the conformal
	identification induced by stereographic projection. In the Euclidean
	unit-ball model, the hyperbolic and spherical metrics are
	\[
	\bar g_{\mathbb H}
	=
	\frac{4}{(1-|x|^2)^2}g_{\mathbb R^{n+1}},
	\qquad
	\bar g_{\mathbb S}
	=
	\frac{4}{(1+|x|^2)^2}g_{\mathbb R^{n+1}},
	\]
	so that
	\[
	\bar g_{\mathbb S}
	=
	(\phi')^2\bar g_{\mathbb H},
	\qquad
	\phi'(r)=\cos r
	=
	\frac{1-|x|^2}{1+|x|^2}.
	\]
	Thus the conformal factor relating the two ambient metrics is exactly
	the radial function $\phi'$ appearing in
	\eqref{def:horoconvex}. Under this conformal change, the transformation law for the
	second fundamental form introduces the normal derivative of the
	conformal factor. Then the hyperbolic horo-convexity inequality
	$h_{\mathbb H}\geq g_{\mathbb H}$ transforms precisely into \eqref{def:horoconvex}.

	This conformal correspondence is reflected by the model
	hypersurfaces: hyperbolic horospheres, represented in the ball model
	by Euclidean spheres tangent to the ideal boundary, correspond under
	stereographic projection to geodesic spheres in the northern
	hemisphere tangent to the equator. Thus the equality models identified
	in \cite[Proposition~2.1]{PanScheuer2025} are the spherical images of
	hyperbolic horospheres, and \eqref{def:horoconvex} may be viewed as
	the conformal spherical analogue of hyperbolic horo-convexity.
	
	\begin{figure}[htbp]
		\centering
		\includegraphics[width=0.55\textwidth, trim=0cm 2cm 0 6cm,
		clip]{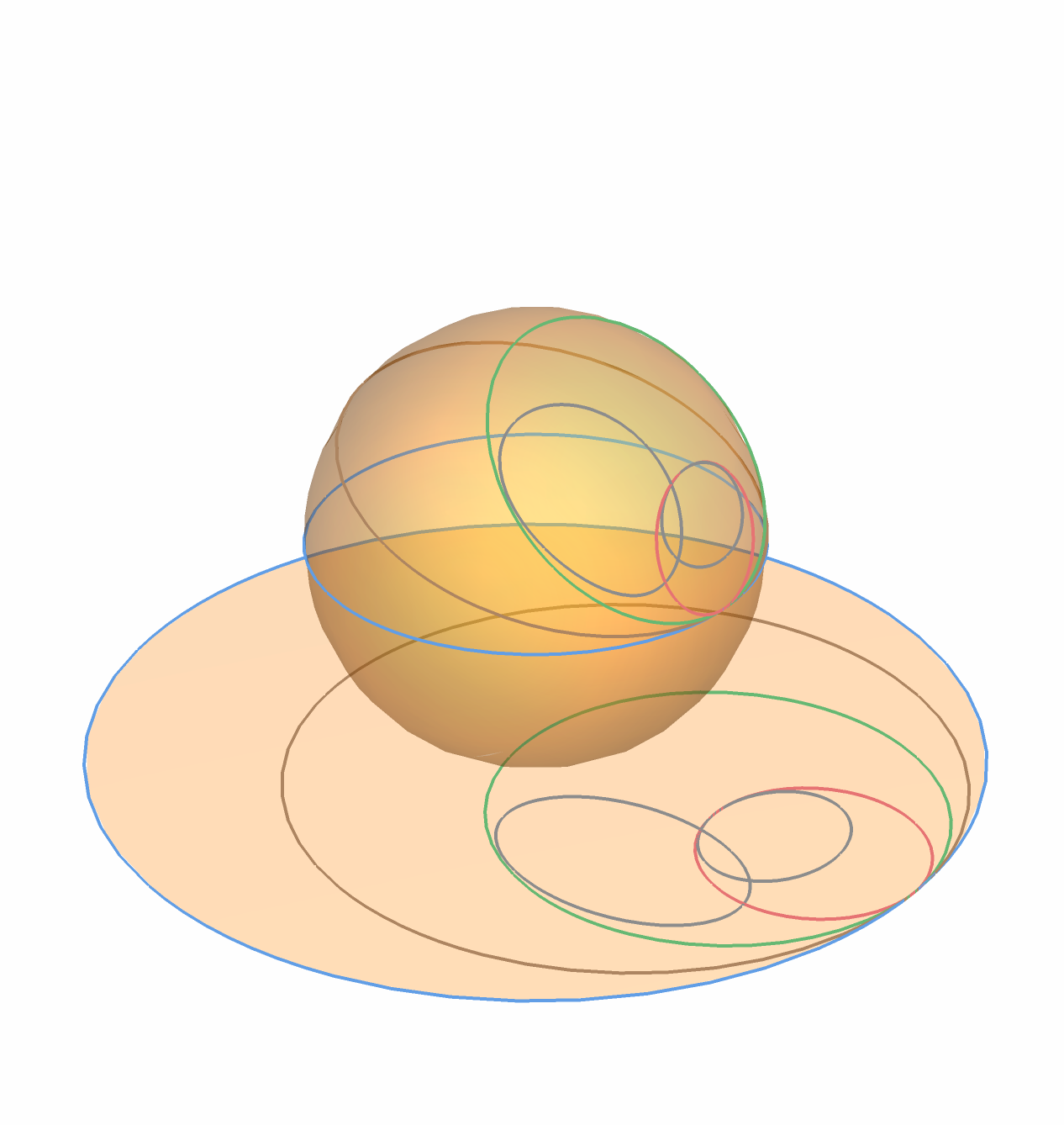}
		\caption{Illustration of the conformal transformation between the definitions of h-convexity in the sphere and the hyperbolic space.}
		\label{fig:nom}
	\end{figure}

	For $k=0,\ldots,n$, let $\sigma_k$ denote the unnormalized
	elementary symmetric polynomial of degree $k$, that is, 
	\[
	\sigma_0:=1,
	\qquad
	\sigma_k(\kappa)
	:=
	\sum_{1\leq i_1<\cdots<i_k\leq n}
	\kappa_{i_1}\cdots\kappa_{i_k}.
	\]
	In particular,
	\(
	\sigma_1=H.
	\) In turn, the normalized elementary symmetric polynomials are 
	\[
	H_k
	:=
	\binom{n}{k}^{-1}\sigma_k.
	\]
	Thus $H_0=1$, and $H_1=\frac{H}{n}$. 
	Throughout the paper, $\sigma_k$ denotes the unnormalized function,
	whereas $H_k$ denotes its normalized counterpart.

	We regard $F$ as a function of the Weingarten operator,
	\[
	F=F(h_i^{j}),
	\]
	or equivalently as a function of $h_{ij}$ and $g_{ij}$. We denote its
	first and second derivatives by
	\[
	\dot F^{ij}
	:=
	\frac{\partial F}{\partial h_{ij}},
	\qquad
	\ddot F^{ij,kl}
	:=
	\frac{\partial^2F}
	{\partial h_{ij}\partial h_{kl}}.
	\]
	At a point at which the Weingarten operator is diagonal, we also write
	\(
	\dot F^i
	:=
	\frac{\partial F}{\partial\kappa_i}.
	\)
	In an orthonormal frame that diagonalizes $h$, it holds
	\(
	\dot F^{ij}=\dot F^i\delta^{ij}.
	\)
	
	The linearized operator associated with $F$ is
	\[
	\Delta_{\dot F}
	:=
	\dot F^{ij}\nabla_i\nabla_j.
	\]
	Strict monotonicity of $F$ implies
	$
	\dot F^i>0
	$
	on $\Gamma_+$, and hence $\Delta_{\dot F}$ is elliptic.
	
	Euler's identity for degree-one homogeneous functions gives
	\begin{equation}\label{eq:Euler-F}
	\dot F^{ij}h_{ij}=F.
	\end{equation}
	Moreover, the first derivatives of $F$ are homogeneous of degree
	zero: $\dot F^i(\lambda\kappa)=\dot F^i(\kappa)$, for any $\lambda>0$.

	We refer to \cite[Section 2.1]{Gerhardt2006} for further details about symmetric curvature functions.

	\subsection{Spherical quermassintegrals}
	
	For a convex domain
	$\Omega\subset\mathbb S^{n+1}$ with smooth boundary, we denote its spherical
	quermassintegrals by $W_0(\Omega),\ldots,W_{n+1}(\Omega)$,
	using the normalization adopted in
	\cite[Section~2]{CabezasRivasScheuer2024}. In particular,
	\[
	W_0(\Omega)=|\Omega|,
	\qquad
	W_1(\Omega)=\frac{|\partial\Omega|}{n+1}.
	\]
	
	For a normal variation $\partial_tx=\mathscr F \nu$,
	the first variation of $W_\ell$ has the form
	\begin{equation}\label{eq:first-variation-Wl}
	\frac{d}{dt}W_\ell(\Omega_t)
	=
	c_{n,\ell}
	\int_{M_t}\mathscr F\,\sigma_\ell\,dV_t,
	\end{equation}
	where $c_{n,\ell}>0$ depends only on the normalization. In
	particular, substitution of \eqref{eq:flow} and
	\eqref{def-mu} gives
	\[
	\frac{d}{dt}W_\ell(\Omega_t)=0.
	\]
	
	We recall the monotonicity of the spherical quermassintegrals with
	respect to inclusion:
	\[
	\Omega_1\subset\Omega_2
	\quad\Longrightarrow\quad
	W_k(\Omega_1)\leq W_k(\Omega_2).
	\]

	\subsection{Radii and geometric distances}
	
	For a convex domain
	$\Omega\subset\mathbb S^{n+1}$, the inner radius is
	\[
	\rho_-(\Omega)
	:=
	\sup\left\{
	\rho>0:
	B_\rho(p)\subset\Omega
	\text{ for some }p\in\Omega
	\right\},
	\]
	and the outer radius is
	\[
	\rho_+(\Omega)
	:=
	\inf\left\{
	\rho>0:
	\Omega\subset B_\rho(p)
	\text{ for some }p\in\mathbb S^{n+1}
	\right\}.
	\]
	A point realizing the inner radius is called an inball center. We
	write
	\(
	S_\rho(q):=\partial B_\rho(q)
	\)
	for the geodesic sphere of radius $\rho$ centered at $q$.
	
	For compact subsets
	$K,L\subset\mathbb S^{n+1}$, their Hausdorff distance with respect to
	the spherical metric is
	\[
	d_{\mathcal H}(K,L)
	:=
	\max\left\{
	\sup_{x\in K}
	\operatorname{dist}_{\mathbb S^{n+1}}(x,L),
	\sup_{y\in L}
	\operatorname{dist}_{\mathbb S^{n+1}}(y,K)
	\right\},
	\]
	where
	$\operatorname{dist}_{\mathbb S^{n+1}}(x,L)
	:=
	\inf_{y\in L}
	\operatorname{dist}_{\mathbb S^{n+1}}(x,y).$

	We shall use the elementary estimate
	$
	\left|
	\operatorname{diam}K-\operatorname{diam}L
	\right|
	\leq
	2d_{\mathcal H}(K,L).
	$
	
	\subsection{Final conventions and short-time existence}
	
	A hypersurface, or a solution of the flow, is called admissible if
	its principal curvature vector belongs to $\Gamma_+$ at every point.
	In the present setting, admissibility is equivalent to strict
	convexity.
	
	Since the initial hypersurface is strictly horo-convex, its principal
	curvatures belong to $\Gamma_+$. By Assumption~\ref{ass:F},
	\[
	\dot F^i>0
	\]
	on $\Gamma_+$, so the local part of \eqref{eq:flow} defines a
	strictly parabolic operator. 
	
	Moreover, the non-local term is well defined near the initial
	embedding. Indeed, since
	$M_0\subset\mathbb S^{n+1}_+$, we have
	$\phi'>0$ on
	$M_0$; in turn,
	strict horo-convexity implies strict convexity, hence
	$\kappa\in\Gamma_+$ and therefore $\sigma_\ell>0$.
	Then
	$
	\int_{M_0}\phi'\sigma_\ell\,dV_0>0.
	$
	By continuity, this denominator remains positive for sufficiently
	small time. Hence $\mu$ defines a smooth non-local functional in a
	neighbourhood of the initial embedding.
	
	The factor $\phi'(r)$ 
	depends only on the position of the hypersurface relative to the
	fixed origin $\mathcal O$. It does not involve the second fundamental form and
	therefore does not contribute to the principal symbol of the
	equation. Likewise, the term $\mu(t)$ depends only on time
	and does not affect parabolicity. 
	
	Therefore, after fixing the tangential parametrization in the
	standard way, \eqref{eq:flow} is a fully nonlinear, strictly
	parabolic equation with a smooth non-local lower-order term. Standard
	short-time existence theory (see \cite{HuiskenPolden1999} or \cite[Section 2.5]{Gerhardt2006}) yields a unique smooth solution on a
	maximal time interval
	$
	[0,T)$, with $T \in(0,\infty]$.

	In the following, we always assume that the flow is defined on its maximal time interval $[0,T)$, unless stated otherwise.
	
	Additionally, $C$ typically denotes a positive constant which may
	change from line to line. It is independent of time and of the point
	on the evolving hypersurface, and may depend on the dimension, the
	initial hypersurface, the fixed curvature function $F$ and the
	preserved quermassintegral. Dependence on $F$ refers only to the fixed structural constants of the
	curvature function, such as bounds for its derivatives on compact
	subsets of $\Gamma_+$ determined by the pinching estimate.
	
	Constants whose precise dependence is relevant will
	be distinguished by subscripts.

	\section{Some useful identities and evolution equations} \label{sec:evol}

	\begin{lemma}[Space derivatives of $\phi'$ and the support function]\label{lemma:deriv-u}
		The support function $u$ and the radial weight $\phi'$ satisfy
		\begin{enumerate}[label=\textup{(\alph*)}]
			\item $\nabla_i u = - h_i^k\nabla_k \phi'$, \smallskip
			\item \label{item:F_Hessian_u_sup_funct} $\Delta_{\dot F} u  = F\phi' -u\dotF^{ij} h^2_{ij} + \phi\,\avg{\partial_r,\nabla F}$, \smallskip
			\item \label{item:norm_grad_cos}$|\nabla\phi'|^2 = \phi^2 |\nabla r|^2=   \phi^2-u^2$.\smallskip
			\item  $\Delta_{\dot F} \phi' = F u- \phi'\, \dotF^{ij} g_{ij}. $
		\end{enumerate}
		
		\begin{proof}
			The first equation can be found in \cite[Lemma 2.6]{GuanLi2015}, where one also finds that
			\[\nabla_i \nabla_j \, u= -g^{k\ell}\nabla_k h_{ij} \nabla_\ell \phi'  + \phi' h_{ij} -u h^2_{ij},\]
			Then (b) follows by contracting the latter with $\dotF^{ij}$, using \eqref{eq:Euler-F}  and realizing that
			\begin{align*}
			-g^{k\ell}\dot{F}^{ij}\nabla_k h_{ij} \nabla_\ell \phi' = -g^{k\ell}\ \nabla_k F \phi'' \nabla_\ell r = \phi\,\avg{\partial_r,\nabla F}.
			\end{align*}
			
			The last identity follows from the relation
			$\nabla \phi'=-\phi\nabla r = -\phi \, \partial_r^\top$. Indeed, we have
			\[
			|\nabla\phi'|^2
			=\phi^2|\nabla r|^2 = \phi^2 \big(1- \avg{\partial_r,\nu}^2\big)
			=\phi^2\left(1-\frac{u^2}{\phi^2}\right)
			=\phi^2-u^2.
			\]
			Finally, (d) follows from \cite[Lemma 2.2]{GuanLi2015}, again after contraction with $\dot F$ and \eqref{eq:Euler-F}.
		\end{proof}
	\end{lemma}

	We next derive the evolution formulas that will be relevant throughout the paper.
	
	\begin{lemma}[Evolution equations on $\mathbb{S}^{n+1}$] \label{evol-eq}
		Under $\pt x=\mathfrak{F} \nu = (\mu(t) \phi' \, - F)\nu$, we have
		
		\begin{enumerate}[label=\textup{(\alph*)}]
			\item $\pt g_{ij}=2\mathfrak{F}  h_{ij} $
			\smallskip
			\item \label{item:evolmetricainv} $\pt g^{ij}=  -2\mathfrak{F} h^{ij}$
			\smallskip
			\item \label{item:evol_u_F} $\partial_t u =\Delta_{\dot F} u  -2F\phi' + u\dotF^{ij} h^2_{ij}+ \mu - \mu u^2$.
			\smallskip
			\item  \label{item:evol_cosr_F} $\pt \phi' \, = \Delta_{\dot F} \phi' \, + \big( \dotF^{ij}g_{ij} - \mu u\big) \phi' \,$
			\medskip
			\item \label{item:evol_hij_F}
			$\displaystyle
			\begin{aligned}[t]
			\partial_t h_{ij} &=
			\Delta_{\dot F} h_{ij}
			+\ddot F^{kl,pq}\nabla_i h_{pq}\nabla_j h_{kl}
			+\left(\mu \phi'-2F\right) h^2_{ij} \\
			&\quad
			+\bigl(
			\dot F^{kl} h^2_{kl}
			-\dot F^{kl}g_{kl}
			-\mu u
			\bigr)h_{ij}
			+2Fg_{ij}.
			\end{aligned}
			$
			\smallskip
			\item
			$\displaystyle
			\begin{aligned}[t]
			\partial_t h_i^j
			={}&  \Delta_{\dot F} h_i^j
			+\ddot F^{kl,pq}\nabla_i h_{pq}\nabla^j h_{kl}
			-\mu \phi' (h^2)_i^j +\bigl(\dot F^{kl} h^2_{kl}-\dot F^{kl}g_{kl}-\mu u\bigr)h_i^j
			+2F\delta_i^j.
			\label{eq:mixed-h-evolution}
			\end{aligned}$ 
		\end{enumerate}

		\begin{proof}
			The first two evolution equations can be deduced from \eqref{eq:flow}
			by a direct computation as it was done in \cite{Huisken84}. We now prove \ref{item:evol_u_F}. Since $u=\langle\phi\,\partial_r,\nu\rangle$, differentiating  yields
			\[
			\partial_t u
			=\langle\onabla_t(\phi\,\partial_r),\nu\rangle
			+\langle\phi\,\partial_r,\onabla_t\nu\rangle.
			\]
			
			Using that $\onabla_t\nu=-\nabla\mathfrak F$ together with
			$\onabla_\nu(\phi\,\partial_r)=\phi'\nu$ (see
			\eqref{eq:conformalfield}), we obtain
			\[
			\begin{aligned}
			\partial_tu
			&=\mathfrak F\phi'
			-\langle\phi\,\partial_r,\nabla\mathfrak F\rangle=(\mu\phi'-F)\phi'
			+\langle\phi\,\partial_r,\nabla F-\mu\nabla\phi'\rangle.
			\end{aligned}
			\]
			Subtracting from this the expression for $\Delta_{\dotF} u$ given in Lemma~\ref{lemma:deriv-u}, we obtain
			\[
			\partial_tu-\Delta_{\dot F}u
			=
			(\mu\phi'-2F)\phi'
			+u\dot F^{ij} h^2_{ij}
			+\mu |\nabla\phi'|^2,
			\]
			and the conclusion follows from Lemma~\ref{lemma:deriv-u} (c).

			We now prove \ref{item:evol_cosr_F}. By the flow equation, we have
			\begin{equation}
			\pt \phi' \, = \phi'' \partial_t r = - \phi \, \mathfrak{F}  \avg{ \nu, \partial_r} = (F -\mu \phi') u, \end{equation}
			hence the conclusion follows from Lemma \ref{lemma:deriv-u} (d). 
			
			We now turn to the proof of \ref{item:evol_hij_F}. We begin with the evolution equation for $h_{ij}$ given by \cite[Theorem~3.2(iv) and Corollary~3.3]{HuiskenPolden1999}.  To explain the
			ambient-curvature notation in that formula, fix a point of the hypersurface
			and choose an  orthonormal frame
			\[
			e_0:=\nu,
			\quad
			e_1,\ldots,e_n\in TM.
			\]
			There holds
			\begin{align*}
			\partial_t h_{ij}
			={}&
			\Delta_{\dot F}h_{ij}
			+
			\ddot F^{kl,pq}\nabla_i h_{kl}\nabla_j h_{pq}
			-
			\mu \nabla_i\nabla_j\phi'       +
			(\mu \phi'-F)
			\left(
			h^2_{ij}-\bar R_{0i0j}
			\right)                               \\
			&+
			\dot F^{kl}
			\left(
			-h_{kl} h^2_{ij}
			+
			h^2_{kj}h_{il}
			-
			h_{kj} h^2_{il}
			+
			h^2_{kl}h_{ij}
			\right)                                                        \\
			&-
			\dot F^{kl}
			\left(
			\bar R_{kli m}h_j^m
			+
			\bar R_{klj m}h_i^m
			+
			\bar R_{mij l}h_k^m
			+
			\bar R_{0i0j}h_{kl}
			-
			\bar R_{0k0l}h_{ij}
			+
			\bar R_{m l j k}h_i^m
			\right).                                                      
			\end{align*}
			We now simplify the cubic terms in \(h\), denoted as $A_{ij}$. Euler's formula \eqref{eq:Euler-F} gives
			\begin{align*}
			A_{ij}
			=
			-F h^2_{ij}
			+
			\dot F^{kl} h^2_{kj}h_{il}
			-
			\dot F^{kl}h_{kj} h^2_{il}
			+
			\dot F^{kl} h^2_{kl}h_{ij} =
			-F h^2_{ij}
			+
			\dot F^{kl} h^2_{kl}h_{ij},
			\end{align*}
			where	the two middle terms cancel because at a fixed point, if we choose an
			orthonormal frame diagonalizing \(h\), both are equal to $\dot F^\ell\kappa_\ell^3\delta_{ij}$ with opposite signs. 
			
			We next simplify the ambient curvature terms $C_{ij}$ from the last line. From \eqref{ambient-R} we get expressions of the form 
			\[	\bar R_{kli m}h_j{}^m
			=
			g_{ki}h_{lj}-g_{li}h_{kj},
			\qquad \text{and} \qquad
			\bar R_{0k0l}=g_{kl}.\]
			Substituting them into the curvature block and cancelling the
			opposite terms gives
			\[
			C_{ij}
			=
			-\dot F^{kl}
			\left\{
			g_{kj}h_{il}
			+
			g_{kl}h_{ij}
			-
			g_{lj}h_{ik}
			\right\} + F g_{ij} =
			-\dot F^{kl}g_{kl}h_{ij}
			+
			F g_{ij},
			\]
			where we have used the symmetry of $\dot F$. Finally, since \(\bar R_{0i0j}=g_{ij}\), we have
			\begin{align*}
			(\mu \phi'-F)
			\left(
			h^2_{ij}-\bar R_{0i0j}
			\right) -\mu \nabla_i\nabla_j\phi'                                                      
			&=
			(\mu \phi'-F) h^2_{ij}
			-
			\mu u h_{ij}
			+
			F g_{ij}.
			\end{align*}
			Collecting all the above computations, the desired formula follows.
			
			From $h^i_j=  g^{ik}h_{kj}$ and the evolution equation in (b), we get
			\[
			\partial_t h_i^j
			=g^{jk}\partial_t h_{ik}-2(\mu\phi'-F)(h^2)_i^j.
			\]
			Consequently, comparing with (e), apart from raising indices the only change is that the total coefficient of $(h^2)_i^j$ is $(\mu\phi'-2F)-2(\mu\phi'-F)=-\mu\phi'$, which leads to (f).
		\end{proof}
	\end{lemma}
	
	\section{Preservation of horo-convexity and curvature pinching}\label{sec:preservation-hconvexity}
	
	\subsection{Evolution of the horo-convexity tensor} \label{sec:S-tensor}
	
	We now introduce the tensor whose nonnegativity is equivalent to
	horo-convexity.  Define
	\begin{equation}\label{eq:def-S}
	S_i^j:=\phi'h_i^j+(u-1)\delta_i^j.
	\end{equation}
	Thus horo-convexity is equivalent to $S_i^j\geq0$. Since $\phi'>0$, the relation between its eigenvalues
	and the principal curvatures can be used without degeneracy.
	
	\begin{lemma}[Evolution of $S_i{}^j$]\label{lem:S-evolution}
		The tensor $S_i{}^j$ satisfies
		\begin{align}
		(\partial_t- \Delta_{\dot F})S_i^j
		={}&\phi'\ddot F^{kl,pq}
		\nabla_i h_{pq}\nabla^j h_{kl}
		-2\dot F^{kl}\nabla_k\phi'\nabla_lh_i^j
		\notag\\
		&-\mu (S^2)_i^j
		+(\dot F^{kl} h^2_{kl}-2\mu)S_i^j
		+\dot F^{kl} h^2_{kl}\delta_i^j.
		\end{align}
	\end{lemma}
	
	\begin{proof}

		Applying the product rule for $ \Delta_{\dot F}$ to \eqref{eq:def-S}, we obtain
		\begin{align}
		(\partial_t- \Delta_{\dot F})S_i^j
		= \phi'(\partial_t- \Delta_{\dot F})h_i^j
		+h_i^j(\partial_t- \Delta_{\dot F})\phi' 
		+(\partial_t- \Delta_{\dot F})u\,\delta_i^j
		-2\dot F^{k\ell}\nabla_k\phi'\nabla_\ell h_i^j.
		\label{eq:S-product-rule}
		\end{align}
		Substituting here  the evolution equations (c), (d) and (f) from Lemma \ref{evol-eq}, after straightforward cancellation of terms of opposite sign and rearranging, yields
		\begin{align*}
		(\partial_t- \Delta_{\dot F})S_i^j
		={}&\phi'\ddot F^{k\ell,pq}
		\nabla_i h_{pq}\nabla^j h_{k\ell}
		-2\dot F^{k\ell}\nabla_k\phi'\nabla_\ell h_i{}^j\\
		&-\mu(\phi')^2(h^2)_i^j
		+\phi'(\dot F^{k\ell} h^2_{k\ell}- 2\mu u)h_i^j
		+(u \dot F^{k\ell} h^2_{k\ell} +\mu-\mu u^2)\delta_i^j.
		\end{align*}
		Setting $v:=1-u$, from \eqref{eq:def-S} we have $\phi'h_i^j=S_i^j+v\delta_i^j$
		and hence
		\[
		(\phi')^2(h^2)_i^j
		=(S^2)_i^j+2vS_i{}^j+v^2\delta_i^j.
		\]
		The coefficient of $S_i^j$ therefore equals
		\[
		-2\mu v+\dot F^{k\ell} h^2_{k\ell}-2\mu u=\dot F^{k\ell} h^2_{k\ell}-2\mu,
		\]
		whereas the coefficient of $\delta_i^j$ is
		\begin{align*}
		-\mu v^2+v(\dot F^{k\ell} h^2_{k\ell}-2\mu u)+u\dot F^{k\ell} h^2_{k\ell}+\mu-\mu u^2
		&=\dot F^{k\ell} h^2_{k\ell}+\mu\bigl(1-(v+u)^2\bigr)=\dot F^{k\ell} h^2_{k\ell}.
		\end{align*}
		This leads to the evolution equation in the statement.
	\end{proof}
	
	\subsection{Preservation of horo-convexity}
	In this section, we prove the main key to our results, namely the preservation of horo-convexity.
	
	First, recall that the initial hypersurface lies in the northern hemisphere $\mathbb{S}^{n+1}_{+}$. We next verify that the evolving hypersurfaces remain in the open
	northern hemisphere. The argument is the fully nonlinear analogue of
	\cite[Corollary~4.4]{CabezasRivasScheuer2024}. We include the details
	needed to justify the application of the strong maximum principle,
	since the zeroth-order coefficient contains the non-local term
	$\mu(t)$.
	
	\begin{prop}[Hemisphere preservation]\label{prop:cos_remains_positive}
		Let $M_t$ be a smooth solution to \eqref{eq:flow} on a maximal time interval $[0,T)$. If $\phi' >0$ on $M_0$, then $\phi' >0$ on $M_t$ for all $t\in[0,T)$.
	\end{prop}

	\begin{proof}
		Suppose that $\phi'$ vanishes, and let $t_0\in(0,T)$ be the first
		time such that
		\(
		\min_{M_{t_0}}\phi'=0.
		\)
		Then
		\[
		\phi'>0
		\quad\text{for }t<t_0,
		\qquad
		\phi'\geq0
		\quad\text{on }M_{t_0}.
		\]
		
		We first show that $\mu$ is bounded on $[0,t_0]$. Since the solution
		is admissible, $\kappa \in \Gamma_+$ and hence
		\(
		\sigma_\ell>0.
		\)
		Moreover, $\phi'$ cannot vanish identically on $M_{t_0}$. Otherwise,
		$M_{t_0}$ would be contained in the equator and, by connectedness,
		would coincide with it, contradicting that $\kappa_i >0$ for all $i$. Therefore,
		\(
		\int_{M_{t_0}}\phi'\sigma_\ell\,dV_{t_0}>0.
		\)
		By continuity, the denominator in \eqref{def-mu} is
		bounded away from zero on a time interval ending at $t_0$. Since the
		numerator is continuous there, $\mu$ is bounded near $t_0$. It is
		also bounded on the preceding compact time interval, where
		$\phi'>0$. 
		
		The coefficient
		\(
		a(x,t):=\dot F^{ij}g_{ij}-\mu u
		\)
		in Lemma \ref{evol-eq} (d) is therefore bounded on
		$M\times[0,t_0]$. Choose
		\(
		\Lambda>\sup_{M\times[0,t_0]}a
		\)
		and set
		\(
		w:=e^{-\Lambda t}\phi'.
		\)
		Then
		\[
		(\partial_t-\Delta_{\dot F})w
		+
		(\Lambda-a)w
		=
		0,
		\]
		where $\Lambda-a>0$. Since $w\geq0$ on
		$M\times[0,t_0]$ and $w(\cdot,0)>0$, the strong maximum principle
		implies that $w$ cannot vanish at time $t_0$. This contradicts the
		definition of $t_0$.
	\end{proof}

	Next, to complete the proof, we need the following lemma, which was proved by Brendle, Choi, and Daskalopoulos \cite[Lemma 5]{BrendleChoiDaskalopoulos2017} and by Choi, Kim, and Lee  \cite[Lemma 4.1]{ChoiKimLee2024}.
	
	\begin{lemma}\label{lemma:eigenvalues_tensor}
		Let $D$ be the multiplicity of the smallest eigenvalue at a point $\xi_0$ on $M_{t_0}$ for $t_0 > 0$ so that 
		\[
		\sigma_1 = \cdots = \sigma_{D}< \sigma_{D+1} \leq \cdots \leq \sigma_n,
		\]
		where $\sigma_1, \cdots, \sigma_n$ are the  eigenvalues of a tensor $\mathcal{P}$. Suppose $\eta$ is a smooth function defined on 
		\[
		\mathcal{M} = \bigcup_{0 < t \le t_0} M_t \times \{t\}
		\]
		such that $\eta \le \sigma_1$ on $\mathcal{M}$ and $\eta= \sigma_1$ at $ (t_0,\xi_0)$. Then, at the point $(t_0,\xi_0)$ with coordinates satisfying 
		$g_{ij} = \delta_{ij}$ and $\mathcal{P}_{ij} = \sigma_i \delta_{ij}$, we have 
		\[\nabla_i \mathcal{P}_{k\ell} = \partial_i \eta \, \delta_{k\ell}, \quad \text{for} \quad 1 \le k,\ell \le D,\]
		and
		\[\nabla_i \nabla_i \eta \le \nabla_i \nabla_i\mathcal{P}^{1}_{1} - 
		\sum_{\alpha> D} \frac{2 (\nabla_i\mathcal{P}^{1}_{\alpha})^2}{\sigma_{\alpha}- \sigma_1}, 
		\qquad
		\partial_t \eta\ge \partial_t \mathcal{P}^1_{1}.\]
	\end{lemma}
	
	\begin{thm}[Preservation of horo-convexity]
		\label{thm:preservation-hconvexity}
		Let $F$ satisfy the hypotheses in Assumption \ref{ass:F}, and let $M_t$ be a smooth solution of \eqref{eq:flow}
		on $[0,T)$. If $M_0$ is horo-convex, then $M_t$ satisfies the same property for every $t\in[0,T)$.
		Equivalently, $S_i^j\geq0$ is preserved under the flow.
	\end{thm}
	
	\begin{proof}
		We first assume that $S_i^j>0$ at $t=0$. Suppose, to the contrary,
		that positivity is lost. Then there is a first time $t_0>0$ and a point
		$\xi_0\in M_{t_0}$ at which the smallest eigenvalue of $S_i^j$ is zero.
		Choose a local orthonormal frame at $(\xi_0,t_0)$ that diagonalizes both
		$h_i^j$ and $S_i^j$. Write
		\[
		h_i^j=\kappa_i\delta_i^j,
		\qquad
		S_i^j=\sigma_i\delta_i^j,
		\]
		where
		\[
		\sigma_1=\cdots=\sigma_D=0
		<\sigma_{D+1}\leq\cdots\leq\sigma_n.
		\]
		As $\sigma_i=\phi'\kappa_i+u-1$,
		we can write
		\begin{equation}
		\phi'\kappa_1=1-u,
		\qquad
		\sigma_\alpha-\sigma_1
		=\phi'(\kappa_\alpha-\kappa_1)
		\quad (\alpha>D).
		\label{eq:null-relations}
		\end{equation}
		
		Let $\eta$ be a smooth lower support for the smallest eigenvalue $\sigma_1$ of
		$S_i^j$ at $(\xi_0,t_0)$;  i.e., $\eta = \sigma_1(t_0, \xi_0)$ and $\eta < \sigma_1$ on $\mathcal{M} = \bigcup_{0 < t \le t_0} M_t \times \{t\}$.
		
		Then by Lemma \ref{lemma:eigenvalues_tensor}, at $(t_0, \xi_0)$ we obtain
		\begin{itemize}
			\item[\textcircled{i}] \label{item:condició_i} $\nabla_k S_j^i = 0, \quad \text{for } 1 \le i, j \le D$,
			\item[\textcircled{ii}]  \label{item:condició_ii}  $\partial_t S_1^1 \le 0, \quad  \Delta_{\dot F}S_1^1
			\geq \displaystyle
			2\sum_{\alpha>D}
			\frac{\dot{F}^k (\nabla_kS^1_\alpha)^2}
			{\sigma_\alpha-\sigma_1}$, where $\dot{F}^k = \frac{\partial F}{\partial \kappa_k}$.
		\end{itemize}

		At the chosen diagonal frame, notice that for $\alpha>D$ and for all $k$ it follows that
		\begin{equation} \label{gradS}
		\nabla_kS_\alpha^1 =
		(\nabla_k\phi')h_\alpha^1
		+\phi'\nabla_kh_\alpha^1
		+(\nabla_ku)\delta_\alpha^1
		=\phi'\nabla_kh_\alpha^1.
		\end{equation}
		Using the Codazzi identity,  \eqref{eq:null-relations} and raising indices at the orthonormal
		frame, we may also write
		\begin{align}
		\Delta_{\dot F}S_1^1
		\geq
		2\phi'\sum_{\alpha>D}
		\frac{{\dot F^k} (\nabla_1h_{k\alpha})^2}
		{\kappa_\alpha-\kappa_1}.
		\label{eq:L-S-lower-bound}
		\end{align}
		
		We now evaluate the evolution equation from Lemma \ref{lem:S-evolution} in an orthonormal frame at $(t_0, \xi_0)$. Since
		$S_1{}^1=0$ and $(S^2)_1^1=0$, we obtain
		\begin{align}
		(\partial_t- \Delta_{\dot F})S_1^1
		={}&\phi'\ddot F^{kl,pq}
		\nabla_1h_{pq}\nabla_1h_{kl} + \Theta, \quad \text{with} \quad \Theta :=
		\dot F^{k} \big(\kappa_k^2 -2\nabla_k\phi'\nabla_k h_{11}\big).
		\label{eq:null-S-evolution}
		\end{align}
		Combining this with \textcircled{ii}  above and 
		\eqref{eq:L-S-lower-bound}  gives
		\begin{align} 
		0\geq{}&
		\phi'\left(
		\ddot F^{kl,pq}\nabla_1h_{pq}\nabla_1h_{kl}
		+2\sum_{\alpha>D}
		\frac{{\dot F^k} (\nabla_1h_{k\alpha})^2}
		{\kappa_\alpha-\kappa_1}
		\right)	+ \Theta.
		\label{eq:contact-master}
		\end{align}

		On the other hand, \textcircled{i} gives
		\begin{equation}\label{eq:gradient-S-null-block}
		0
		=
		(\nabla_k\phi')h_i^j
		+\phi'\nabla_kh_i^j
		+(\nabla_ku)\delta_i^j.
		\end{equation}
		for every $k$ and $1\leq i,j\leq D$.
		If $i\neq j$, then, in the chosen diagonal frame, 
		$\phi'\nabla_kh_i^j=0$.
		Since $\phi'>0$, we conclude that
		\[
		\nabla_kh_a^b=0,
		\qquad a\neq b,\quad a,b\leq D.
		\]
		Suppose now that $i=j$, and take $k=1$ in
		\eqref{eq:gradient-S-null-block}. Then it follows that
		\[
		\begin{aligned}
		0
		&=
		\kappa_1\nabla_1\phi'
		+\phi'\nabla_1h_i^i
		+\nabla_1u =
		\phi'\nabla_1h_i^i,
		\end{aligned}
		\]
		where we have applied Lemma \ref{lemma:deriv-u} (a), which guarantees that $\nabla_1 u=-\kappa_1\nabla_1\phi'$.
		Thus $\nabla_1h_i^i=0$.

		We use the differential characterization of inverse concavity together
		with the vanishing relations at the minimum point.

		It follows that the only possibly nonzero components of
		$B_{ij}:=\nabla_1h_{ij}$ are $B_{\alpha\beta}$, with
		$\alpha,\beta>D$, and $B_{1\alpha}=B_{\alpha1}$, with $\alpha>D$.
		
		Let $\widehat B$ denote the transverse part of $B$, defined by
		$\widehat B_{\alpha\beta}:=B_{\alpha\beta}$ for $\alpha,\beta>D$, and
		$\widehat B_{ij}:=0$ otherwise. Applying the differential
		characterization of inverse concavity to $\widehat B$, we obtain
		\begin{equation}\label{eq:inverse-concavity-transverse}
		\Upsilon
		:=
		\ddot F^{kl,pq}\widehat B_{pq}\widehat B_{kl}
		+
		2\sum_{\alpha,\beta>D}
		\frac{\dot F^\beta}{\kappa_\alpha}
		\widehat B_{\beta\alpha}^2
		\geq
		\frac{2}{F}
		\left(
		\sum_{\alpha>D}
		\dot F^\alpha\widehat B_{\alpha\alpha}
		\right)^2
		\geq0.
		\end{equation}
		Since $\kappa_\alpha>\kappa_1>0$ for every $\alpha>D$, we have
		\[
		\frac{1}{\kappa_\alpha-\kappa_1}
		-
		\frac{1}{\kappa_\alpha}
		=
		\frac{\kappa_1}
		{\kappa_\alpha(\kappa_\alpha-\kappa_1)}
		>0.
		\]
		Therefore,
		\begin{align}
		\ddot F^{kl,pq}\widehat B_{pq}\widehat B_{kl}
		+
		2\sum_{\alpha,\beta>D}
		\frac{\dot F^\beta\widehat B_{\beta\alpha}^2}
		{\kappa_\alpha-\kappa_1}
		=
		\Upsilon
		+
		2\sum_{\alpha,\beta>D}
		\dot F^\beta
		\frac{\kappa_1}
		{\kappa_\alpha(\kappa_\alpha-\kappa_1)}
		\widehat B_{\beta\alpha}^2
		\geq0.
		\label{eq:transverse-expression-nonnegative}
		\end{align}
		
		We next recall the spectral second-derivative formula. At a diagonal
		matrix $h_i^j=\kappa_i\delta_i^j$, for every symmetric tensor $B$,
		\begin{equation}\label{eq:spectral-second-derivative-F}
		\ddot F^{kl,pq}B_{pq}B_{kl}
		=
		\ddot F^{ij}B_{ii}B_{jj}
		+
		2\sum_{i<j}
		\frac{\dot F^i-\dot F^j}
		{\kappa_i-\kappa_j}B_{ij}^2, \qquad \text{where} \quad 
		\ddot F^{ij}:=\frac{\partial^2F}{\partial\kappa_i\partial\kappa_j}.
		\end{equation}
		Taking into account all the components of $B$ that vanish, the above formula 
		reduces to
		\begin{equation}\label{eq:Hessian-transverse-mixed}
		\ddot F^{kl,pq}B_{pq}B_{kl}
		=
		\ddot F^{kl,pq}\widehat B_{pq}\widehat B_{kl}
		+
		2\sum_{\alpha>D}
		\frac{\dot F^1-\dot F^\alpha}
		{\kappa_1-\kappa_\alpha}B_{1\alpha}^2.
		\end{equation}
		
		Similarly, since $B_{k\alpha}$ can be nonzero only if $k=1$ or
		$k>D$, we can write
		\begin{align}
		2\sum_{\alpha>D}
		\frac{\dot F^kB_{k\alpha}^2}
		{\kappa_\alpha-\kappa_1}
		={}&
		2\sum_{\alpha>D}
		\frac{\dot F^1B_{1\alpha}^2}
		{\kappa_\alpha-\kappa_1}
		+
		2\sum_{\alpha,\beta>D}
		\frac{\dot F^\beta\widehat B_{\beta\alpha}^2}
		{\kappa_\alpha-\kappa_1}.
		\label{eq:correction-decomposition}
		\end{align}

		Adding \eqref{eq:Hessian-transverse-mixed} and
		\eqref{eq:correction-decomposition}, we obtain
		\begin{align}	
		\ddot F^{kl,pq}B_{pq}B_{kl}
		+
		2\sum_{\alpha>D}
		\frac{\dot F^kB_{k\alpha}^2}
		{\kappa_\alpha-\kappa_1}
		=  
		\ddot F^{kl,pq}\widehat B_{pq}\widehat B_{kl}
		+
		2\sum_{\alpha,\beta>D}
		\frac{\dot F^\beta\widehat B_{\beta\alpha}^2}
		{\kappa_\alpha-\kappa_1}
		+	2\sum_{\alpha>D}
		\frac{\dot F^\alpha}
		{\kappa_\alpha-\kappa_1}B_{1\alpha}^2,
		\label{eq:full-gradient-decomposition}
		\end{align}
		where the first two terms on the right-hand side are nonnegative by
		\eqref{eq:transverse-expression-nonnegative}. The last term is also
		nonnegative, since $\dot F^\alpha>0$ and
		$\kappa_\alpha>\kappa_1$. We conclude that
		\begin{align} 
		&
		\ddot F^{kl,pq}\nabla_1h_{pq}\nabla_1h_{kl}
		+
		2\sum_{\alpha>D}
		\frac{\dot F^k(\nabla_1h_{k\alpha})^2}
		{\kappa_\alpha-\kappa_1}
		\geq0.
		\label{eq:full-gradient-expression-positive}
		\end{align}
		
		It remains to estimate the last two terms. With this aim, we take $i = j = 1$ in \eqref{eq:gradient-S-null-block}, which gives
		\[
		\kappa_1\nabla_k\phi'
		+\phi'\nabla_kh_{11}
		= -\nabla_ku = h_k^l\nabla_l\phi'
		\]
		Consequently, as our frame diagonalizes $h_i^j$, it follows that
		\begin{equation}\label{eq:gradient-h11-contact}
		\phi'\nabla_kh_{11}
		=	(\kappa_k-\kappa_1)\nabla_k\phi'.
		\end{equation}
		Therefore,
		\begin{align}
		\Theta
		=
		\dot F^k
		\left(
		\kappa_k^2
		-
		\frac{2}{\phi'}
		(\kappa_k-\kappa_1)w_k^2\right), \quad \text{where } \quad w_k:=|\nabla_k\phi'|.
		\label{eq:remaining-terms}
		\end{align}
		
		Setting $v:=1-u=\phi'\kappa_1$, for every $k$ we have
		\begin{align}
		&
		\kappa_k^2
		-
		\frac{2}{\phi'}
		(\kappa_k-\kappa_1)w_k^2
		=
		\kappa_k^2
		-
		\frac{2w_k^2}{\phi'}\kappa_k
		+
		\frac{2vw_k^2}{(\phi')^2}
		=
		\left(
		\kappa_k-\frac{w_k^2}{\phi'}
		\right)^2
		+
		\frac{w_k^2}{(\phi')^2}
		\left(2v-w_k^2\right).
		\label{eq:complete-square-final}
		\end{align}
		
		We now use Lemma \ref{lemma:deriv-u} (c) to deduce
		\[
		w_k^2\leq|\nabla\phi'|^2\leq 	1-u^2
		=
		(1-u)(1+u) \leq
		2 v.
		\]
		Both terms on the right-hand side of
		\eqref{eq:complete-square-final} are therefore nonnegative. Accordingly, we reach
		\[\Theta \geq \dot F^1\kappa_1^2.\]

		At the point $(t_0, \xi_0)$, since the hypersurface lies in the open northern hemisphere,
		$u\leq\phi<1$ as well as $\phi'>0$,
		and therefore
		\[
		\kappa_1=\frac{1-u}{\phi'}>0,
		\]
		Moreover, the strict monotonicity of $F$ implies $\dot F^1>0$, and hence $\Theta > 0$. Plugging this and \eqref{eq:full-gradient-expression-positive} into \eqref{eq:contact-master} leads to a contradiction. It follows that, if $S_i^j$ is positive definite initially, its
		smallest eigenvalue cannot vanish at a positive time. Therefore, $S_i^j>0$
		is preserved along the flow.
		
		Finally, suppose only that $S\geq0$ at $t=0$. One then uses the
		standard perturbation
		\[
		\widetilde S_i^j =
		S_i^j+\varepsilon e^{\lambda t}\delta_i{}^j,
		\]
		where $\lambda$ is chosen sufficiently large on each compact time
		interval $[0,T']\subset[0,T)$. The strict argument applies to the
		perturbed tensor. Letting $\varepsilon\downarrow0$ gives
		\[
		S_i^j\geq0
		\qquad\text{on }M_t,
		\quad 0\leq t<T,
		\]
		and hence horo-convexity is preserved.
	\end{proof}

	\subsection{Curvature pinching}
	
	The preservation result gives a pointwise lower bound for every
	principal curvature in terms of the support function, but this bound
	still depends on the position of the hypersurface. For the geometric
	and curvature estimates in the following sections, we need the
	stronger, scale-invariant comparison between the largest and smallest
	principal curvatures established below.
	
	\begin{lemma}\label{lem:curvature-pinching}
		There is a constant $C_0>0$, depending only on $M_0$, such that
		\[
		\kappa_n\leq C_0\kappa_1
		\]
		for all $t \in [0, T)$.
	\end{lemma}
	
	\begin{proof}
		As horo-convexity implies strict convexity, we can choose $0<\varepsilon\leq1/n$ such that
		$P_{ij}:=h_{ij}-\varepsilon Hg_{ij}$ is positive definite on $M_0$.
		The aim is to prove that $P_{ij}\geq0$ is preserved.
		
		Taking the trace of the evolution equation for $h_i^j$ from Lemma \ref{evol-eq} (f), we obtain
		\[
		\partial_tH
		=
		\Delta_{\dot F}H
		+
		\ddot F^{kl,pq}\nabla_mh_{pq}\nabla^mh_{kl}
		-
		\mu(t)\phi'|A|^2
		+
		\bigl(\dot F^{kl}h^2_{kl}-\dot F^{kl}g_{kl}-\mu(t)u\bigr)H
		+
		2nF.
		\]
		Combining this identity with
		the evolution equations for $h_{ij}$ and $g_{ij}$ gives
		\begin{align}
		(\partial_t-\Delta_{\dot F})P_{ij}
		={}&
		\ddot F^{kl,pq}\nabla_ih_{pq}\nabla_jh_{kl}
		-
		\varepsilon g_{ij}
		\ddot F^{kl,pq}\nabla_mh_{pq}\nabla^mh_{kl}
		\notag\\
		&+
		(\mu(t)\phi'-2F) P^2_{ij}
		+
		\bigl(\dot{F}^{kl} h^2_{kl}-\dot F^{kl}g_{kl}-\mu(t)u-2\varepsilon HF\bigr)P_{ij}
		\notag\\
		&+
		\varepsilon\mu(t)\phi'
		\bigl(|A|^2-\varepsilon H^2\bigr)g_{ij}
		+
		2F(1-n\varepsilon)g_{ij}.
		\label{eq:P-evolution}
		\end{align}
		
		Suppose that $P_{ij}\geq0$ and that $P_{ij}$ has a null eigenvector
		at some point. Choose a principal orthonormal frame such that this
		vector is $e_1$. Then $P_{11}=0$ and
		$\kappa_1=\varepsilon H$. Accordingly, the terms containing $P_{ij}$ or
		$P^2_{ij}$ vanish in the null direction.
		
		Moreover,
		\[
		|A|^2-\varepsilon H^2
		\geq
		\left(\frac1n-\varepsilon\right)H^2
		\geq0.
		\]
		Since $\mu(t)\geq0$, $\phi'>0$, $F>0$ and
		$1-n\varepsilon\geq0$, all the zeroth-order terms in the null
		direction are nonnegative.
		
		It remains to consider the gradient terms. Since $F$ is concave and
		inverse concave, the pinching result for curvature functions
		\cite[Theorem~4.1]{AndrewsPinching} gives
		\begin{align}
		&
		\ddot F^{kl,pq}\nabla_1h_{pq}\nabla_1h_{kl}
		-
		\varepsilon
		\ddot F^{kl,pq}\nabla_mh_{pq}\nabla^mh_{kl}
		+
		2\sup_\Lambda
		\dot F^{kl}
		\left(
		2\Lambda_k^p\nabla_lP_{1p}
		-
		\Lambda_k^p\Lambda_l^qP_{pq}
		\right)
		\geq0.
		\label{eq:P-gradient-condition}
		\end{align}
		Thus the null-eigenvector condition of the tensor maximum principle
		is satisfied, and $P_{ij}\geq0$ is preserved.
		
		Consequently, $\kappa_1\geq\varepsilon H\geq\varepsilon\kappa_n$. The claim follows with $C_0=\varepsilon^{-1}$.
	\end{proof}
	
	We have proved that the flow preserves both horo-convexity and a
	uniform curvature pinching estimate. Horo-convexity ensures that the
	evolving hypersurfaces remain strictly convex in the open northern
	hemisphere, while the pinching estimate gives a uniform comparison
	between all principal curvatures. In the next section, we combine
	these properties with the preservation of the quermassintegral to
	obtain uniform inner and outer radius estimates and, subsequently,
	bounds for the global term.
	
	\section{Geometric consequences of horo-convexity and pinching} \label{sec:geometric-estimates}
	
	We now derive the geometric estimates needed to control the global
	term and to obtain the upper curvature bound. Recall that
	horo-convexity is preserved under the flow (cf. Theorem \ref{thm:preservation-hconvexity}) and that we have proved the curvature
	pinching estimate from Lemma \ref{lem:curvature-pinching}.

	Since the solution remains in the open northern hemisphere,
	horo-convexity implies strict convexity. Indeed,
	$\phi'\kappa_i\geq1-u$, while $\phi'>0$ and $u\leq\phi<1$. Hence
	$\kappa_i>0$ for every $i$.
	
	\subsection{Inner and outer radius estimates}
	
	We first recall the geometric consequences of the pinching estimate.
	These results are independent of the particular evolution equation
	and can be applied to each hypersurface $M_t$.
	
	\begin{lemma}\label{lem:radius-estimates}
		There are positive constants $d_1$ and $d_2$, depending only on $n$,
		the pinching constant $C_0$ and the preserved quermassintegral
		$W_\ell(\Omega_0)$, such that
		\[
		d_1\leq\rho_-(\Omega_t)\leq\rho_+(\Omega_t)
		\leq\frac{\pi}{2}-d_2
		\]
		for every $t\in[0,T)$.
	\end{lemma}
	
	\begin{proof}
		Since the pinching estimate has already been proved and
		$W_\ell(\Omega_t)=W_\ell(\Omega_0)$ is preserved, we may apply
		\cite[Proposition~3.1 and Corollary~3.2]
		{CabezasRivasScheuer2024} to each $M_t$. These results give
		$\rho_+(\Omega_t)\leq C\rho_-(\Omega_t)$, where $C$ depends only on
		$n$ and $C_0$, together with a positive lower bound for the inner
		radius and an upper bound for the outer radius strictly below
		$\pi/2$. The resulting constants are independent of $t$.
	\end{proof}
	
	\subsection{Bounds for the global term}
	
	We next use the radius estimates to control the denominator in the
	definition \eqref{def-mu} of the global term. 
	\begin{lemma}\label{lem:weighted-curvature-integrals}
		Let
		$M \subset \Snp$ be a closed horo-convex hypersurface enclosing a bounded domain
		$\Om$.	For every $j=0,\ldots,n$, there are constants $c_j,C_j>0$, depending
		only on $n$, $C_0$ and $W_\ell(\Omega)$, such that
		\[
		c_j\leq\int_{M}\phi'\sigma_j\,dV
		\leq\int_{M}\sigma_j\,dV\leq C_j.
		\]
	\end{lemma}
	
	\begin{proof}
		We first consider $j=1$. Since $\sigma_1=H$, taking the trace of the
		horo-convexity inequality \eqref{def:horoconvex} and integrating gives
		\begin{equation}\label{eq:weighted-H-first}
		\int_{M}\phi'H\,dV
		\geq n|M|-n\int_{M}u\,dV.
		\end{equation}
		By the divergence theorem, we have
		\begin{equation}\label{eq:divergence-conformal-field}
		\int_{M}u\,dV
		=
		\int_{M}\langle \phi \partial_r,\nu\rangle\,dV
		=
		(n+1)\int_{\Omega}\phi'\,dV.
		\end{equation}
		Combining \eqref{eq:weighted-H-first} and
		\eqref{eq:divergence-conformal-field}, we obtain
		\begin{equation}\label{eq:weighted-H-domain}
		\int_{M}\phi'H\,dV
		\geq
		n|M|-n(n+1)\int_{\Omega}\phi'\,dV.
		\end{equation}
		
		Let $B_r(\mathcal O)$ be the geodesic ball centered at the north pole $\mathcal O$ with
		$|B_r(\mathcal O)|=|\Om|$. Since $\ph'(r)=\cos r$ is decreasing in $r$, a rearrangement
		argument gives
		\[
		\int_{\Omega} \phi' \, dV
		=
		\int_{\Omega \cap B_r(\mathcal{O})} \phi' \, dV
		+
		\int_{\Omega \setminus B_r(\mathcal{O})} \phi' \, dV
		\le
		\int_{\Omega \cap B_r(\mathcal{O})} \phi' \, dV
		+
		\int_{B_r(\mathcal{O}) \setminus \Omega} \phi' \, dV
		=
		\int_{B_r(\mathcal{O})} \phi' \, dV.
		\]
		Moreover, by the spherical isoperimetric inequality,
		\[
		|M|\ge |\partial B_r(\mathcal O)|.
		\]
		Consequently, as for the coordinate sphere $\partial B_r(\mathcal O)$ one has $u=\ph(r)$, we can write
		\begin{align*}
		\int_M \ph'H \dd V
		&\ge n|\partial B_r(\mathcal O)|-n(n+1)\int_{B_r(O)}\ph'\dd V  \\
		&=n \int_{\partial B_r(\mathcal O)}(1-u)\dd V
		= n(1-\ph(r))|\partial B_r(\mathcal O)| .
		\end{align*}

		The radius estimates from Lemma \ref{lem:radius-estimates} imply that $|\Omega|$ is bounded from above
		and away from zero. Thus the radius $r$ remains in a compact
		subinterval of $(0,\pi/2)$. Accordingly,
		\[
		(1-\ph(r))|\partial B_r(O)|
		=(1-\sin r)\om\sin^n r
		\]
		is bounded from below by a positive constant depending only on
		$n,C_0$ and $W_\ell(\Om)$. Hereafter $\om$ denotes the area of the unit $n$-sphere. Hence
		\[
		\int_M \ph'H \dd V\ge c_1>0 .
		\]
		
		We now pass from $H_1 = H/n$ to $H_j$, $1\le j\le n$. By the pinching condition, we have
		\[
		H_1=\frac1n\sum_{i=1}^n\kappa_i
		\le \kappa_n
		\le C_0\kappa_1 .
		\]
		On the other hand, every $j$-fold product of principal curvatures is at least
		$\kappa_1^j$, and hence, because $H_j$ is the average of all such products, $H_j\ge \kappa_1^j$.
		Thus
		\[
		H_1\le C_0 H_j^{1/j} .
		\]
		Multiplying by $\ph'$ and applying H\"older's inequality, we obtain
		\begin{align*}
		\int_M \ph'H_1\dd V
		&\le C_0\int_M \ph' H_j^{1/j}\dd V = C_0\int_M (\ph'H_j)^{1/j}(\ph')^{(j-1)/j}\dd V \\
		&\le C_0
		\left(\int_M \ph'H_j\dd V\right)^{1/j}
		\left(\int_M \ph'\dd V\right)^{(j-1)/j} .
		\end{align*}
		
		Since $\Omega\subset\mathbb S^{n+1}_+$, the monotonicity of
		$W_1$ with respect to inclusion gives
		\[
		\frac{|M|}{n+1}
		=
		W_1(\Omega)
		\leq
		W_1(\mathbb S^{n+1}_+)
		=
		\frac{\omega_n}{n+1},
		\]
		where $\omega_n=|\mathbb S^n|$. Hence
		\begin{equation} \label{area-bound}
		|M|\leq\omega_n,
		\end{equation}

		Using  
		$ \int_M \ph'\dd V\le |M|\le \omega_n$, we get
		\[
		c_1\le C_0\om^{(j-1)/j}
		\left(\int_M \ph'H_j\dd V\right)^{1/j} .
		\]
		Therefore
		\[
		\int_M \ph'H_j\dd V
		\ge \frac{c_1^j}{C_0^j\om^{j-1}}
		=:c_j>0 .
		\]
		This proves the desired lower bound for every $j=1,\ldots,n$.

		It remains to deal with the case $j=0$. Applying the preceding H\"older estimate with
		$j=n$ gives
		\[
		c_1\le C_0
		\left(\int_M \ph'H_n\dd V\right)^{1/n}
		\left(\int_M \ph'\dd V\right)^{(n-1)/n} .
		\]
		On the other hand, since $0\le \ph'\le 1$ and by the Gauss--Bonnet formula on the sphere (see e.g. \cite{Solanes}), we reach
		\[
		\int_M \ph'H_n\dd V\le \int_M K\dd V\le C(n),
		\]
		from which	we conclude that
		\[
		\int_M \ph'\dd V
		\ge \left(\frac{c_1}{C_0 C(n)^{1/n}}\right)^{n/(n-1)}
		=:c_0>0.\]
		
		It remains to establish the upper bounds. For $j=0$, the assertion
		follows from $\sigma_0=1$ and the area estimate \eqref{area-bound}. For $1\leq j\leq n$, the pinching estimate gives
		$
		H_j
		\leq
		\kappa_n^j
		\leq
		C_0^j\kappa_1^j.
		$
		Moreover,
		$
		H_n
		=
		\kappa_1\cdots\kappa_n
		\geq
		\kappa_1^n,
		$
		and therefore
		$
		H_j\leq C_0^jH_n^{j/n}.
		$
		Since $H_n=K$ is the Gauss--Kronecker curvature, the
		Gauss--Bonnet formula cited above and
		H\"older's inequality, together with the area bound
		\eqref{area-bound}, yield
		\[
		\begin{aligned}
		\int_M H_j\,dV
		\leq
		C_0^j
		\int_M H_n^{j/n}\,dV
		\leq
		C_0^j
		\left(
		\int_M H_n\,dV
		\right)^{j/n}
		|M|^{1-j/n}
		&\leq
		C_j.
		\end{aligned}
		\]
		Since
		$
		\sigma_j=\binom{n}{j}H_j
		$
		and $0<\phi'\leq1$, we conclude that
		\[
		\int_M\phi'\sigma_j\,dV
		\leq
		\int_M\sigma_j\,dV
		\leq
		C_j
		\]
		for every $j=0,\ldots,n$. This completes the proof.
	\end{proof}
	
	The preceding estimate gives the required control of the non-local
	term.
	
	\begin{cor}\label{cor:global-term-by-max-F}
		There is a constant $C>0$, depending only on $n$, $C_0$ and
		$W_\ell(\Omega_0)$, such that
		\[
		\mu(t)\leq C\max_{M_t}F
		\]
		for every $t\in[0,T)$.
	\end{cor}
	
	\begin{proof}
		By the definition of $\mu(t)$,
		\[
		\mu(t)
		\leq
		\max_{M_t}F
		\frac{\int_{M_t}\sigma_\ell\,dV_t}
		{\int_{M_t}\phi'\sigma_\ell\,dV_t}.
		\]
		The numerator is uniformly bounded and the denominator is uniformly
		bounded away from zero by
		Lemma~\ref{lem:weighted-curvature-integrals}, because the constants there depend on $\ell, n, C_0$ and $W_\ell(\Omega_t) = W_\ell(\Omega_0)$.
	\end{proof}
	
	The radius estimates have thus been used to control the denominator
	of the global term. We now use the same lower bound for the inner
	radius for a second purpose: to choose a point with respect to which
	the support function remains uniformly positive on a time interval of
	fixed length.
	
	\subsection{A lower bound for the support function}
	
	The lower bound for the inner radius allows us to choose, at every
	initial time, a point with respect to which the support function stays
	uniformly positive on a time interval of fixed length.
	
	\begin{lemma}\label{lem:interior-ball}
		Let $t_0\in[0,T)$ and let $p$ be the center of an inball of
		$\Omega_{t_0}$. There is a constant $\tau_0>0$, independent of
		$t_0$, such that
		\[
		B_{d_1/4}(p)\subset\Omega_t
		\]
		for every $t\in[t_0,\min\{t_0+\tau_0,T\})$.
	\end{lemma}
	
	\begin{proof}
		By Lemma~\ref{lem:radius-estimates},
		$B_{d_1}(p)\subset\Omega_{t_0}$. The result follows from the
		interior-ball comparison argument in
		\cite[Proposition~5.2]{CabezasRivasScheuer2024}.
		
		We only verify the point in that argument which depends on the
		curvature function. A geodesic sphere of radius $R$ has all principal
		curvatures equal to $\cot R$. By the homogeneity and normalization of
		$F$,
		\[
		F(\cot R,\ldots,\cot R)
		=
		\cot R\,F(1,\ldots,1)
		=
		n\cot R.
		\]
		Thus the shrinking geodesic spheres considered in
		\cite[Proposition~5.2]{CabezasRivasScheuer2024} have the correct local
		curvature speed for the present flow. To see that they remain inner
		barriers, suppose that one of these spheres touches $M_t$ for the
		first time from the inside. At the contact point, the second
		fundamental form of the inner sphere is greater than or equal to that
		of $M_t$. Since $F$ is monotone increasing, the corresponding
		curvature functions satisfy
		\[
		F_{\mathrm{sphere}}\geq F_{M_t}.
		\]
		In turn, the outward normal speeds at the contact point are
		\(
		-F_{\mathrm{sphere}}\)
		and 
		\(\mu(t)\phi'-F_{M_t},
		\)
		respectively. Since $\mu(t)\geq0$ and $\phi'>0$, we have
		\[
		-F_{\mathrm{sphere}}
		\leq
		-F_{M_t}
		\leq
		\mu(t)\phi'-F_{M_t}.
		\]
		Thus the inner sphere cannot cross $M_t$ at a first contact point.
		This is precisely the comparison argument used in
		\cite[Proposition~5.2]{CabezasRivasScheuer2024}, and therefore the rest of
		that proof applies unchanged.
		
		The radius of the comparison sphere satisfies an ordinary
		differential equation whose initial value is $d_1/2$. Hence one can
		choose $\tau_0>0$, depending only on $n$ and $d_1$, such that the
		radius remains at least $d_1/4$ on
		$[t_0,t_0+\tau_0]$. In particular, $\tau_0$ is independent of the
		initial time $t_0$.
	\end{proof}
	
	The preceding result gives a uniform lower bound for the distance from
	$p$ to $M_t$ on an interval of fixed length. The corresponding
	support-function estimate is now an immediate consequence of the
	geometric result in \cite[Lemma~2.1]{CabezasRivasScheuer2024}.
	
	Let $\hat r_p$ denote the distance from $p$ and define
	\[
	\hat u_p
	:=
	\left\langle
	\phi(\hat r_p)\partial_{\hat r_p},\nu
	\right\rangle.
	\]
	\vspace*{-0.4cm}
	\begin{lemma}\label{lem:support-lower-bound}
		There is a constant $d_3>0$, independent of $t_0$, such that the
		support function with respect to the point $p$ chosen in
		Lemma~\ref{lem:interior-ball} satisfies
		\[
		\hat u_p-d_3\geq d_3
		\]
		on $M_t$ for every
		$t\in[t_0,\min\{t_0+\tau_0,T\})$.
	\end{lemma}
	
	\begin{proof}
		By \cite[Lemma~2.1]{CabezasRivasScheuer2024}, we have the lower bound
		\[
		\hat u_p
		\geq
		\phi\bigl(\operatorname{dist}(p,M_t)\bigr) \geq\phi\left(\frac{d_1}{4}\right),
		\]
		because Lemma~\ref{lem:interior-ball} gives
		$\operatorname{dist}(p,M_t)\geq d_1/4$. Hence the conclusion follows with
		$d_3:=\frac12\phi(d_1/4)$.
	\end{proof}
	
	We have thus obtained constants $d_3>0$ and $\tau_0>0$, independent of
	$t_0$, such that for every $t_0\in[0,T)$ one can choose a point
	$p\in\Omega_{t_0}$ for which
	\[
	\hat u_p-d_3\geq d_3
	\]
	on a time interval of length $\tau_0$. We now combine this estimate
	with the bound for the global term obtained in the previous section.
	
	\section{Upper estimates for the curvature} \label{sec:upper-curvature}
	
	A simplification of the curvature argument in
	\cite{CabezasRivasScheuer2024} is available in the present setting.
	Indeed, Corollary~\ref{cor:global-term-by-max-F} gives an upper bound for $\mu(t)$ in terms of $\max_{M_t}F$. Therefore, at a maximum of the Tso auxiliary function \cite{Tso}, all positive
	terms containing $\mu(t)$ are at most quadratic, whereas the leading
	negative curvature term is cubic.
	
	\begin{prop}\label{prop:upper-curvature-estimate}
		There is a constant $C>0$, depending only on $M_0$, $n$ and the
		preserved quermassintegral, such that
		\[
		F+H+|A|+\mu(t)\leq C
		\]
		on $M_t$ for every $t\in[0,T)$.
	\end{prop}
	
	\begin{proof}
		Fix $t_0\in[0,T)$ and choose $p\in\Omega_{t_0}$ as in
		Lemma~\ref{lem:support-lower-bound}. On the interval
		\[
		I_{t_0}
		:=
		[t_0,\min\{t_0+\tau_0,T\}),
		\]
		we suppress the dependence on $p$ and write
		$\hat u=\hat u_p$, $\hat\phi=\phi(\hat r_p)$ and
		$\hat\phi'=\phi'(\hat r_p)$. Consider the Tso-type auxiliary function
		\[
		\Psi=\frac{F}{\hat u-d_3}
		\]
		The starting point is the evolution equation
		\begin{equation}\label{Phi-evolution-exact}
		\begin{aligned}
		\big(\partial_t - \Delta_{\dot F}\big)\Psi
		={}&\frac{2}{\hat u-d_3}{\dot F}^{ij}\nabla_i\Psi\nabla_j\hat u 
		+ \Psi\left(\dot F^{ij}g_{ij}-\frac{d_3}{\hat u-d_3}\dot F^{ij} h^2_{ij}\right)
		+2\Psi^2 \hat\phi' \\
		& -\mu(t)\Psi u 
		-\frac{\mu(t)}{\hat u-d_3}
		\left(
		\phi' \dot F^{ij}h^2_{ij}
		+\Psi\big(
		\hat \phi'\phi'
		+\hat \phi \, \phi\langle\nabla\hat r,\nabla r\rangle
		\big)
		\right).
		\end{aligned}
		\end{equation}

		Hereafter set $\mathscr G(t):=\max_{M_t}\Psi$. Since $F=(\hat u-d_3)\Psi$ and $\hat u\leq1$, we have
		$\max_{M_t}F\leq\mathscr G$. Corollary~\ref{cor:global-term-by-max-F}
		therefore gives
		\begin{equation}\label{eq:mu-by-mathscr-G}
		\mu(t)\leq C\mathscr G.
		\end{equation}
		
		We next record a consequence of the curvature pinching estimate.
		Since
		\[
		\kappa_1\leq\kappa_i\leq C_0\kappa_1,
		\]
		the normalized curvature 
		\(
		\widetilde\kappa:=\frac{\kappa}{\kappa_1}
		\)
		belongs to the compact set
		\[
		K_{C_0}
		:=
		\left\{
		\lambda\in\Gamma_+:
		\lambda_1=1,\quad
		1\leq\lambda_i\leq C_0
		\right\}.
		\]
		Since $F$ is homogeneous of degree one, each derivative $\dot F^i$
		is homogeneous of degree zero. Hence
		$\dot F^i(\kappa)=\dot F^i(\widetilde\kappa)$. By the smoothness and
		strict monotonicity of $F$, and by the compactness of $K_{C_0}$, there
		are constants $0<c\leq C$ such that
		\(
		c\leq\dot F^i\leq C
		\)
		for every $i$. Using an orthonormal frame,
		$\dot F^{ij}=\dot F^i\delta^{ij}$, and therefore
		\begin{equation}\label{eq:trace-Fdot-bound}
		\dot F^{ij}g_{ij}
		=
		\sum_{i=1}^n\dot F^i
		\leq C.
		\end{equation}
		Moreover, Euler's identity,  the Cauchy--Schwarz inequality  and \eqref{eq:trace-Fdot-bound} give
		\[
		F^2
		=
		\left(F^i\kappa_i\right)^2
		\leq\sum_{i=1}^n\dot F^i\sum_{i=1}^n\dot F^i\kappa_i^2\leq
		C
		\dot F^i\kappa_i^2.
		\]
		Consequently,
		\begin{equation}\label{eq:weighted-curvature-lower}
		\dot F^{ij}h^2_{ij}
		=
		\sum_i\dot F^i\kappa_i^2
		\geq cF^2 \qquad \text{for some } \quad c>0.
		\end{equation}
		
		Let $(t,\xi)$ be a point where $\Psi$ attains its spatial maximum at
		time $t\in I_{t_0}$. If $\mathscr G$ is not differentiable at $t$, we
		use its upper Dini derivative. The evolution equation of $\Psi$ gives
		\begin{align}
		D^+\mathscr G
		\leq{}&
		\mathscr G\dot F^{ij}g_{ij}
		-
		\frac{d_3 \mathscr G}{\hat u-d_3}
		\dot F^{ij} h^2_{ij}
		+
		2\mathscr G^2\hat\phi'
		-
		\mu \mathscr G u
		-
		\frac{\mu\mathscr G}
		{\hat u-d_3}
		\left(\hat\phi'\phi'
		+
		\hat\phi\phi
		\langle\nabla\hat r,\nabla r\rangle
		\right),
		\label{eq:mathscr-G-maximum}
		\end{align}
		where we have discarded the term
		\(
		-\frac{\mu(t)\phi'}
		{\hat u-d_3}
		\dot F^{ij}h^2_{ij}
		\), which is nonpositive
		because $\mu(t)\geq0$, $\phi'>0$ and
		$\dot F^{ij}h^2_{ij}>0$.
		
		All the remaining ambient quantities are uniformly bounded:
		$|\phi|,|\phi'|,|\hat\phi|,|\hat\phi'|\leq1$ and
		$|\langle\nabla\hat r,\nabla r\rangle|\leq1$. Using
		$\hat u-d_3\geq d_3$ and \eqref{eq:mu-by-mathscr-G}, we obtain
		\[
		\left|
		\mu \mathscr G u
		+
		\frac{\mu	\mathscr G}
		{\hat u-d_3}
		\left(
		\hat\phi'\phi'
		+
		\hat\phi\phi
		\langle\nabla\hat r,\nabla r\rangle
		\right)
		\right|
		\leq C\mathscr G^2.
		\]
		Together with \eqref{eq:trace-Fdot-bound}, this shows that all the
		remaining positive terms in \eqref{eq:mathscr-G-maximum} are bounded
		by $C(\mathscr G+\mathscr G^2)$.

		On the other hand, the leading negative term is cubic. Indeed, by
		\eqref{eq:weighted-curvature-lower} and
		$F=(\hat u-d_3)\mathscr G$, we have
		\[
		-
		\frac{d_3}{\hat u-d_3}
		\mathscr G\dot F^{ij} h^2_{ij}
		\leq
		-cd_3(\hat u-d_3)\mathscr G^3
		\leq
		-cd_3^2\mathscr G^3,
		\]
		where we have also applied Lemma \ref{lem:support-lower-bound}. In short, we conclude that there are positive constants $C_1$ and $C_2$,
		independent of $t_0$ and of the point $p$, such that
		\begin{equation}\label{eq:mathscr-G-cubic}
		D^+\mathscr G
		\leq
		C_1(\mathscr G+\mathscr G^2)-C_2\mathscr G^3.
		\end{equation}
		For $\mathscr G\geq1$, we have
		$\mathscr G+\mathscr G^2\leq2\mathscr G^2$. 
		Then, whenever $\mathscr G\geq 	\Upsilon :=\max\left\{1,\frac{4C_1}{C_2}\right\}$, it holds
		\[
		C_1(\mathscr G+\mathscr G^2)
		\leq
		2C_1\mathscr G^2
		\leq
		\frac{C_2}{2}\mathscr G^3.
		\]
		Thus \eqref{eq:mathscr-G-cubic} gives
		\begin{equation}\label{eq:mathscr-G-decay}
		D^+\mathscr G
		\leq
		-\frac{C_2}{2}\mathscr G^3 \qquad \text{whenever } \quad \mathscr G\geq \Upsilon.
		\end{equation}
		Integrating the differential inequality, we obtain
		\begin{equation}\label{eq:mathscr-G-smoothing}
		\mathscr G(t)
		\leq
		\max\left\{
		\Upsilon,\,
		\frac{1}{\sqrt{C_2(t-t_0)}}
		\right\},
		\qquad t>t_0,
		\end{equation}
		unless the estimate obtained from the initial value
		$\mathscr G(t_0)$ is smaller. In particular, if
		$t_0+\tau_0/2<T$, then
		\[
		\mathscr G\left(t_0+\frac{\tau_0}{2}\right)
		\leq
		\max\left\{
		\Upsilon,\,
		\sqrt{\frac{2}{C_2\tau_0}}
		\right\}.
		\]
		The right-hand side is independent of $t_0$, of the choice of $p$
		and of the initial value $\mathscr G(t_0)$.
		
		We now obtain a bound on the whole maximal time interval. On
		$[0,\min\{\tau_0/2,T\})$, we apply the preceding maximum-principle
		argument with $t_0=0$. Since $M_0$ is smooth and
		$\hat u-d_3\geq d_3$, the initial value $\mathscr G(0)$ is finite.
		Hence $\mathscr G$ is bounded on this initial interval.
		
		Now let $t\in[\tau_0/2,T)$. Set
		\(
		t_0:=t-\frac{\tau_0}{2}
		\)
		and choose $p$ to be the center of an inball of $\Omega_{t_0}$.
		Lemma~\ref{lem:support-lower-bound} ensures that
		$\hat u_p-d_3\geq d_3$ on the interval containing $[t_0,t]$.
		Applying \eqref{eq:mathscr-G-smoothing} at time $t$ gives
		\(
		\mathscr G(t)
		\leq
		\max\left\{
		\Upsilon,\,
		\sqrt{\frac{2}{C_2\tau_0}}
		\right\}.
		\)
		Therefore, after enlarging the constant to include the initial time
		interval,
		\(
		\mathscr G(t)\leq C
		\)
		for every $t\in[0,T)$.
		
		Since $F=(\hat u-d_3)\Psi$ and $\hat u\leq1$, we conclude that
		\(
		F\leq C.
		\)
		It remains to translate this estimate into bounds for all principal
		curvatures. By monotonicity, homogeneity and normalization,
		\[
		F(\kappa_1,\ldots,\kappa_n)
		\geq
		F(\kappa_1,\ldots,\kappa_1)
		=
		n\kappa_1.
		\]
		Using the curvature pinching estimate, we obtain
		\[
		H\leq n\kappa_n\leq nC_0\kappa_1\leq C_0F.
		\]
		Hence $H\leq C$. Since all principal curvatures are positive,
		$|A|\leq H \leq C$. Finally, Corollary~\ref{cor:global-term-by-max-F} gives
		\(
		\mu(t)\leq C\max_{M_t}F\leq C
		\)
		which completes the proof.
	\end{proof}
	
	We have thus obtained uniform bounds for the curvature and the global
	term on the whole maximal existence interval. The essential point is
	that the lower support-function estimate is available on time
	intervals of uniform length, while the cubic differential inequality
	for $\mathscr G$ produces a regularizing bound independent of the
	value of the auxiliary function at the beginning of each interval.
	
	\section{Long-time existence and exponential convergence}
	\label{sec:long-time-convergence}

	\label{sec:long-time-convergence}
	
	We now complete the proof of the main theorem. The upper curvature
	estimate prevents the formation of finite-time singularities. We then
	improve the curvature pinching estimate and obtain exponential decay
	of the traceless second fundamental form. This implies that the
	evolving hypersurfaces approach the family of geodesic spheres at an
	exponential rate. Finally, we use the exponential closeness to geodesic spheres and the preserved quermassintegral to identify the limiting radius. We then adapt the asymptotic argument in \cite[Section~7]{CabezasRivasScheuer2024} to identify the center of the limiting sphere. Once smooth convergence to the centered sphere has been established, a maximum-principle argument for the radial graph gives the exponential rate.
	
	Since all the geometric and curvature estimates obtained above hold
	globally, the changes of origin used in
	\cite{CabezasRivasScheuer2024} to recover these estimates on successive
	time intervals are not needed here.

	\subsection{Long-time existence}
	
	We first show that the estimates obtained in the previous sections
	prevent the formation of finite-time singularities.
	
	\begin{prop}[Long-time existence]
		\label{prop:long-time-existence}
		Let $M_t$ be the smooth solution of \eqref{eq:flow} with strictly 
		horo-convex initial hypersurface $M_0$. Then the solution exists smoothly for
		all time. 
	\end{prop}
	
	\begin{proof}
		Let $[0,T)$ be the maximal interval of existence. We argue by
		contradiction and suppose that $T<\infty$. The estimates obtained in Proposition 
		\ref{prop:upper-curvature-estimate} give
		\(
		|A|+\mu(t)\leq C
		\)
		on $[0,T)$. In particular, all the principal curvatures remain
		uniformly bounded from above.
		
		By the curvature pinching estimate from Lemma \ref{lem:curvature-pinching},
		the normalized curvature vectors
		$\kappa/\kappa_1$ remain in a fixed compact subset of $\Gamma_+$. Since $F$ is strictly monotone and each derivative
		$\dot F^i$ is homogeneous of degree zero, there are constants
		$0<\lambda\leq\Lambda$ such that
		\[
		\lambda|\xi|^2
		\leq
		\dot F^{ij}\xi_i\xi_j
		\leq
		\Lambda|\xi|^2
		\]
		for every covector $\xi$ and every $t\in[0,T)$. Thus the linearized
		operator remains uniformly elliptic.
		
		We next sketch how to represent the flow by a uniformly parabolic
		scalar equation. Fix $t_0\in[0,T)$. By
		Lemma~\ref{lem:interior-ball}, there is a point
		$p\in\Omega_{t_0}$ such that
		\(
		B_{d_1/4}(p)\subset\Omega_t
		\)
		for
		\(
		t\in 
		I_{t_0}:=	[t_0,\min\{t_0+\tau_0,T\}).
		\)
		Since $M_t$ is strictly convex and contains $p$ in its interior, it is
		star-shaped with respect to $p$. 
		
		We may therefore write
		\[
		M_t
		=
		\left\{
		\bigl(\hat r(\theta,t),\theta\bigr):
		\theta\in\mathbb S^n
		\right\},
		\]
		where $\hat r$ denotes the spherical distance from $p$. Set
		\(
		\hat v
		:=
		\sqrt{
			1+
			\frac{|D\hat r|_{\sigma}^2}
			{\phi^2(\hat r)}
		}.
		\)
		Here $D$ and $\sigma$ denote the covariant derivative and the round
		metric on $\mathbb S^n$, respectively. The outward unit normal of the
		radial graph is
		\[
		\nu
		=
		\frac{1}{\hat v}
		\left(
		\partial_{\hat r}
		-
		\frac{\hat \sigma^{ij} \partial_j r}{\phi^2(\hat r)}\partial_i
		\right), \quad \text{and hence } \quad 
		\langle\nu,\partial_{\hat r}\rangle
		=
		\frac{1}{\hat v}.
		\]

		After composing the embeddings with a suitable family of
		diffeomorphisms of $\mathbb S^n$, the tangential component of the
		velocity may be removed. The normal component of
		$\partial_t x=\partial_t\hat r\,\partial_{\hat r}$ is then
		$\partial_t\hat r/\hat v$. Consequently, the radial function satisfies
		\begin{equation}\label{eq:radial-flow}
		\partial_t\hat r
		=
		\hat v\bigl(\mu(t)\phi'(r)-F\bigr).
		\end{equation}
		Notice that $\hat r$ is the distance from the auxiliary point $p$,
		whereas $r$ is the distance from the fixed origin $\mathcal O$ used in the
		definition of the flow. Thus $\phi'(r)$ in \eqref{eq:radial-flow} is a smooth function of
		$(\theta,\hat r)$ through the corresponding point of the sphere. Its
		contribution is multiplied by the graph factor $\hat v$, which
		depends on $D\hat r$, but the forcing term does not involve
		$D^2\hat r$ and hence does not affect the principal part.

		The induced metric of the radial graph is
		\(
		g_{ij}
		=
		\hat r_i\hat r_j+\phi^2(\hat r)\sigma_{ij},
		\)
		while its second fundamental form is
		\[
		h_{ij}
		=
		\frac{1}{\hat v}
		\left(
		-D_iD_j\hat r
		+
		\phi(\hat r)\phi'(\hat r)\sigma_{ij}
		+
		2\frac{\phi'(\hat r)}{\phi(\hat r)}
		\hat r_i\hat r_j
		\right),
		\]
		with the sign convention used throughout the paper. Therefore, the
		principal part of \eqref{eq:radial-flow} is obtained by composing
		$F$ with the Hessian $D^2\hat r$.
		
		The inclusion $B_{d_1/4}(p)\subset\Omega_t$, together with the outer
		radius estimate, gives uniform upper and lower bounds for $\hat r$.
		The support-function estimate
		\[
		\hat u_p
		=
		\left\langle
		\phi(\hat r)\partial_{\hat r},\nu
		\right\rangle
		=
		\frac{\phi(\hat r)}{\hat v}
		\geq 2d_3
		\]
		also gives a uniform bound for $\hat v$ and hence for
		$|D\hat r|_\sigma$. Thus the radial graphs remain uniformly
		star-shaped on every interval $I_{t_0}$.
		
		Since the curvature ratios stay in a compact subset of
		$\Gamma_+$, the ellipticity constants of $F$ are uniformly bounded
		from above and away from zero. Together with the uniform bounds for
		$\hat r$, $D\hat r$, $A$ and $\mu(t)$, this shows that
		\eqref{eq:radial-flow} is a uniformly parabolic fully nonlinear
		equation on every interval $I_{t_0}$, with constants independent of
		$t_0$.
		
		For fixed $(\theta,\hat r,D\hat r)$, the Weingarten operator depends
		affinely on $D^2\hat r$, with principal part
		$-\hat v^{-1}D^2\hat r$. Since $F$ is concave in the Weingarten
		operator, the second-order operator $-\hat vF$ appearing on the
		right-hand side of \eqref{eq:radial-flow} is convex in the Hessian.
		The non-local forcing term is of lower order and does not affect
		either this structural property or uniform parabolicity. Hence, the
		interior estimates of \cite{TianWang} now give a parabolic
		$C^{2,\alpha}$-estimate for $\hat r$ on every smaller
		time interval. In particular,
		its constant  is independent of $t_0$ and $T$. 
		
		To continue with the classical Schauder estimates on the fixed finite
		interval $[0,T)$, we next show that the curvature
		vectors stay in a
		compact subset of $\Gamma_+$, the compact set obtained here is allowed
		to depend on $T$. 
		
		The evolution equation for $\phi'$, together with the pinching
		estimate and the uniform bound for $\mu$, implies by the parabolic
		minimum principle that $\phi'\geq c_T>0$
		on the finite interval $[0,T)$. Indeed, the coefficient in the evolution equation for $\phi'$ is
		bounded from below on $[0,T)$, and hence the claimed estimate follows
		by applying the minimum principle to $e^{Ct}\phi'$.
		Using $u\leq\phi$,
		$\phi^2+(\phi')^2=1$, and horo-convexity, we obtain
		\[
		\kappa_i
		\geq
		\frac{1-u}{\phi'}
		\geq
		1-\sqrt{1-c_T^2}
		=:\delta_T>0.
		\]
		Together with the upper curvature estimate, this places the curvature
		vectors
		in a compact set $K_T\Subset\Gamma_+$. Consequently, the derivatives of \(F\) of every order are uniformly bounded on $K_T$; that is, for every
		$q\geq1$,
		\[
		\sup_{\kappa\in K_T}|\dot{F}^{(q)}(\kappa)|<\infty,
		\]
		where $\dot{F}^{(q)}$ is the $q$-th derivative of $F$ at $\kappa$.
		These bounds generally depend on $T$, because $c_T$ and $\delta_T$ do. In short, since $\hat r$ is uniformly bounded in $C^{2,\alpha}$ and $\dot F^{(q)}$ is bounded on $K_T$ for every $q$, the coefficients of the linearized equation are uniformly $C^\alpha$ on each graphical interval.  This is precisely the Hölder
		regularity required for the parabolic Schauder estimates. Iterating
		them yields, for every $m\geq2$,
		\[
		\|\hat r\|_{C^{m,\alpha}}
		\leq C_{m,T};
		\]
		see~\cite[Chapter~4]{Lieberman}. The constants are independent of the
		reference time $t_0$ but may depend on the fixed finite time $T$ through
		$K_T$. Applying these estimates on overlapping graphical intervals gives
		\[
		\sup_{0\leq t<T}\sup_{M_t}|\nabla^mA|
		\leq C_{m,T}
		\qquad\text{for every }m\geq0.
		\]

		In particular, the embeddings $x(t,\cdot)$ converge smoothly, after
		a time-dependent repara\-metrization, to a smooth limiting embedding as
		$t\nearrow T$.  Thus the smooth limiting embedding at $t=T$ is strictly convex and
		the equation remains uniformly parabolic. The denominator defining
		the global term remains positive by
		Lemma~\ref{lem:weighted-curvature-integrals}. The short-time
		existence theorem for fully nonlinear curvature flows can therefore
		be applied with this limiting hypersurface $M_T$ as initial data, extending
		the solution beyond $T$; see
		\cite[Sections~2.5 and~2.6]{Gerhardt2006}. This contradicts the
		maximality of $T$. Therefore,
		\(
		T=\infty.
		\)
		\end{proof}
	
	\subsection{Exponential decay of the pinching deficit}
	
	We next improve the pinching estimate. 
	
	\begin{lemma}\label{lem:improved-pinching}
		There are constants $C>0$ and $\delta>0$ such that the traceless second fundamental form satisfies
		\[
		|\mathring A|\leq Ce^{-\delta t} \qquad \text{for all} \quad t\geq0.
		\]
	\end{lemma}
	
	\begin{proof}
		Let
		\(
		\varepsilon_0
		:=
		\min_{M_0}\frac{\kappa_1}{H}>0
		\)
		and define
		\(
		\varepsilon(t)
		:=
		\frac1n-
		\left(\frac1n-\varepsilon_0\right)e^{-\delta t}.
		\)
		Consider the tensor
		\[
		P_{ij}:=h_{ij}-\varepsilon(t)Hg_{ij}.
		\]
		At $t=0$, we have $P_{ij}\geq0$.
		
		Its evolution equation is the same as in the proof of
		Lemma~\ref{lem:curvature-pinching}, with $\varepsilon$ replaced by
		$\varepsilon(t)$ and with the additional term
		$-\varepsilon'(t)Hg_{ij}$.  At a null eigenvector of $P_{ij}$, the
		terms containing $P_{ij}$ or $P^2_{ij}$ vanish. The gradient terms
		satisfy the null-eigenvector condition by
		\cite[Theorems~3.2 and~4.1]{AndrewsPinching}.
		
		It remains to estimate the 	zeroth-order terms. First, as
		$0<\varepsilon_0\leq\frac1{n}$ and thus
		$0<\varepsilon(t)\leq\frac1n$, we get
		
		\[
		|A|^2-\varepsilon(t)H^2
		\geq
		\left(\frac1n-\varepsilon(t)\right)H^2
		\geq0.
		\]
		Since $\mu(t)\geq0$ and $\phi'>0$, it follows that
		the coefficient of $g_{ij}$  is bounded from below by
		\[
		2F(1-n\varepsilon(t))-\varepsilon'(t)H = 	\bigg(\frac1n-\varepsilon_0\bigg)e^{-\delta t}
		(2nF-\delta H).
		\]
		
		The pinching estimate already proved gives
		$\kappa_1\geq\varepsilon_0H$. By monotonicity, homogeneity and
		normalization of $F$,
		\[
		F\geq F(\kappa_1,\ldots,\kappa_1)=n\kappa_1 \geq n \varepsilon_0 H.
		\]
		Choosing $0<\delta\leq n^2\varepsilon_0$, we reach
		\[
		2nF-\delta H
		\geq
		(2n^2\varepsilon_0-\delta)H
		\geq0.
		\]
		Consequently, the complete coefficient of $g_{ij}$ in
		the evolution of $P_{ij}$ is nonnegative.

		Therefore,  the tensor maximum principle gives
		\(h_{ij}\geq\varepsilon(t)Hg_{ij}\).
		Consequently,
		\[
		\frac{\kappa_1}{H}
		\geq
		\frac1n-
		\left(\frac1n-\varepsilon_0\right)e^{-\delta t}. \quad \text{Hence} \quad \frac{H}{n}-\kappa_1
		\leq
		\left(\frac1n-\varepsilon_0\right)e^{-\delta t}H.
		\]
		Since
		\(
		\left|\kappa_i-\frac Hn\right|
		\leq\kappa_n-\kappa_1 \leq \sum_{i = 1}^n (\kappa_i -\kappa_1)
		= H-n\kappa_1,
		\)
		it follows that
		\[
		\begin{aligned}
		|\mathring A|^2
		&=
		\sum_{i=1}^n
		\left(\kappa_i-\frac{H}{n}\right)^2\leq
		n^3
		\left(\frac1n-\varepsilon_0\right)^2
		e^{-2\delta t}H^2,
		\end{aligned}
		\]
		and the conclusion follows using the upper curvature estimate $H\leq C$ from Proposition \ref{prop:upper-curvature-estimate}.
	\end{proof}

	The exponential decay of the traceless second fundamental form gives
	a quantitative estimate for the distance of $M_t$ from the family of
	geodesic spheres. 
	\begin{cor}[Quantitative spherical closeness]
		\label{cor:quantitative-spherical-closeness}
		There are geodesic spheres
		$S_{\rho(t)}(q(t))\subset\mathbb S^{n+1}$ and constants
		$C,\delta>0$ such that
		\begin{align}\label{Hd-es}
		d_{\mathcal H}
		\bigl(M_t,S_{\rho(t)}(q(t))\bigr)
		\leq
		Ce^{-\delta t}
		\end{align}
		for all sufficiently large $t$. Moreover, after increasing the
		initial time if necessary, $M_t$ is a normal graph over
		$S_{\rho(t)}(q(t))$, and the corresponding graph functions converge
		to zero exponentially in $C^2$.
	\end{cor}
	
	\begin{proof}
		Fix $p>n$. We apply
		\cite[Theorem~3.1 and Corollary~3.3]
		{RothScheuerAlmostUmbilical} in stereographic coordinates from the south pole to identify the northern
		hemisphere with the Euclidean unit ball and write the round metric
		as
		\[
		\bar g=e^{2\psi}\widetilde g,
		\qquad
		e^\psi=\frac{2}{1+|x|^2}=1+\phi'.
		\]
		In particular,
		$
		\|\psi\|_{L^\infty}\leq\log2.
		$
		
		Let $\widetilde A$ and $\widetilde H$ denote the Euclidean second
		fundamental form and the Euclidean mean curvature,
		respectively. The conformal transformation law gives
		\[
		\widetilde H
		=
		e^\psi H-n\widetilde\nu(\psi)
		=
		(1+\phi') H+nu.
		\]
		Since $M_t$ is strictly horo-convex, we have 
		\(
		\phi' H\geq n(1-u)
		\)
		and therefore 
		$$\widetilde H\geq H+n\geq n.$$
		Since the conformal factor is bounded from above and away from zero,
		the Euclidean and spherical volume forms, and hence the corresponding
		$L^p$-norms, are uniformly equivalent. Using Lemma~\ref{lem:weighted-curvature-integrals}, the upper area
		bound and Hölder's inequality, we obtain
		\[
		\|\widetilde H\|_{L^p(\widetilde M_t)}
		\geq c_p>0.
		\]
		Now the conformal transformation law for the principal curvatures is
		$
		\widetilde\kappa_i
		=
		(1+\phi')\kappa_i+u.
		$
		Thus the stereographic images are strictly convex, and the uniform
		curvature and area estimates give
		\[
		\|\widetilde A\|_{L^p(\widetilde M_t)}
		\leq C.
		\]
		All other quantities on which the constants in the cited results
		depend are uniformly controlled.
		
		In turn, Lemma~\ref{lem:improved-pinching} ensures that
		$
		\|\mathring A\|_{L^p(M_t)}
		\leq
		Ce^{-\delta t}.
		$
		Consequently, for all sufficiently large $t$, the smallness
		condition
		\[
		\|\mathring A\|_{L^p(M_t)}
		\leq
		\varepsilon_0
		\|\widetilde H\|_{L^p(\widetilde M_t)}
		\]           
		in \cite[Theorem~3.1]{RothScheuerAlmostUmbilical} is satisfied.
		Therefore, there is a geodesic sphere
		$S_{\rho(t)}(q(t))$ such that
		\[
		d_{\mathcal H}
		\bigl(M_t,S_{\rho(t)}(q(t))\bigr)
		\leq
		C
		\frac{
			\|\mathring A\|_{L^p(M_t)}^\alpha
		}{
			\|\widetilde H\|_{L^p(\widetilde M_t)}^\alpha
		}
		\leq
		Ce^{-\alpha\delta t}.
		\]
		After replacing $\alpha\delta$ by $\delta$, this proves the Hausdorff
		estimate. 
		
		Finally, strict convexity and
		\cite[Corollary~3.3]{RothScheuerAlmostUmbilical} provide the normal
		graph representation and control of the graph factor. More precisely, \cite[Corollary~3.3]{RothScheuerAlmostUmbilical}
		gives
		$
		v\leq \exp\bigl(Ce^{-\delta t}\bigr).
		$
		Since $e^s-1\leq Cs$ for small $s\geq0$, it follows that
		\[
		v-1\leq Ce^{-\delta t}
		\]
		for all sufficiently large $t$. Since
		$
		v^2=1+\bar g^{ij}\partial_i f_t \partial_j f_t,
		$
		this implies exponential decay of $Df_t$, possibly after decreasing
		the decay rate. Together with the Hausdorff estimate, which controls
		$\|f_t\|_{C^0}$, we obtain
		\[
		\|f_t\|_{C^1}\leq Ce^{-\delta' t}  \qquad \text{	for some \quad $\delta'>0$.}
		\]
		Combining this estimate with the uniform $C^{2,\alpha}$-bounds and interpolation yields exponential convergence of the graph functions $f_t$ in $C^2$. 
	\end{proof}
	
	Now, we are prepared to show that the solution hypersurface $M_t$ of the flow \eqref{eq:flow} with strictly 
	horo-convex initial hypersurface $M_0$ has uniform higher-order estimates which are time independent. Moreover, the graph functions above converge to zero exponentially in any order.
	\begin{cor}\label{uniform-est}
		Let $M_t$ be the smooth solution of \eqref{eq:flow} with strictly 
		horo-convex initial hypersurface $M_0$. Then, for every $m\geq0$, there is a constant $C_m>0$, depending only on the initial hypersurface and the preserved quermassintegral, such that 
		\[\sup_{t\geq0}\sup_{M_t}\vert\nabla^mA\vert\leq C_m.\]
		Moreover, for sufficiently large time $t$ and every integer $k\geq0$, the corresponding graph functions $f_t$ in Corollary \ref{cor:quantitative-spherical-closeness} satisfy 
		\begin{align*}
		\vert f_t\vert_{C^k}\leq Ce^{-\frac{\delta}{2} t},
		\end{align*}
		where $\delta$ is the positive constant from Corollary~\ref{cor:quantitative-spherical-closeness}.
	\end{cor}
	\begin{proof}
		First, from the Hausdorff-distance estimate \eqref{Hd-es} and Lemma \ref{lem:radius-estimates}, for all sufficiently large $t$ we have 
		\begin{align}\label{est-rho}
		0<\frac{d_1}{2}\leq\rho(t)\leq\frac{\pi}{2}-\frac{d_2}{2}.
		\end{align}
		Since $f_t$ converge to zero exponentially in $C^2$, there holds
		\begin{align}\label{decay pf curavture}
		\vert h^i_{j}-\cot\rho(t)\delta^i_{j}\vert\leq C\| f_t\|_{C^2}\leq Ce^{-\delta t}.
		\end{align}
		Combining this with \eqref{est-rho}, for sufficiently large $t$, it follows that
		\begin{align*}
		\kappa_i(x,t)\geq\frac{1}{2}\tan\frac{d_2}{2}>0, \quad \forall x\in M_t.
		\end{align*}
		Therefore, together with the uniform upper curvature estimate, we obtain that the curvature vectors stay in a fixed compact subset of $\Gamma_+$. Now, the standard Schauder estimates give 
		\[\sup_{t\geq0}\sup_{M_t}\vert\nabla^mA\vert\leq C_m \]
		for every $m\geq0$, where $C_m>0$ is time independent. Thus,  after transferring
		the estimates to the uniformly controlled graph coordinates, we have
		\(
		\|f_t\|_{C^m}\leq C_m.
		\) Finally, using interpolation and the decay of $\|f_t\|_{C^0}$, we obtain
		\[
		\|f_t\|_{C^k}\leq C\|f_t\|^{1-\frac{k}{2k}}_{C_0}\|f_t\|^{\frac{1}{2}}_{C^{2k}}\leq C_ke^{-\frac{\delta}{2} t}, \]
where $D_-$ denotes the lower Dini derivative.		
		for every $k\geq0$.
	\end{proof}

	\subsection{Convergence to a sphere} \label{subsec:convergence-to-sphere} We conclude the proof by identifying the limiting sphere and establishing the rate of convergence. We first adapt the argument of \cite[Section~7]{CabezasRivasScheuer2024} to show that every subsequential limiting sphere is centered at the fixed origin. Since we do not assume that $\mathcal O \in \Omega_t$, we work at a maximum point of the radial distance instead of using a global radial parametrization with respect to $\mathcal O$. This gives smooth convergence to $S_{r_\infty}(\mathcal O)$. The exponential rate, which is not obtained in \cite[Section~7]{CabezasRivasScheuer2024}, is then proved separately by writing the hypersurfaces as radial graphs over $S_{r_\infty}(\mathcal O)$ and applying a maximum-principle argument together with interpolation.
	
	\begin{thm}[Exponential convergence] \label{thm:exponential-convergence} There is a unique radius $r_\infty\in(0,\pi/2)$ satisfying $W_\ell(B_{r_\infty}(\mathcal O)) = W_\ell(\Omega_0)$ such that $M_t$ converges exponentially in $C^\infty$ to the geodesic sphere $S_{r_\infty}(\mathcal O)$ centered at the fixed origin. More precisely, for all sufficiently large $t$, the hypersurface $M_t$ can be written as a radial graph over $S_{r_\infty}(\mathcal O)$, \[ M_t = \left\{ \bigl(r_\infty+f(\theta,t),\theta\bigr): \theta\in\mathbb S^n \right\}, \] and, for every integer $m\geq0$, there are constants $C_m,\delta_m>0$ such that \[ \|f(\cdot,t)\|_{C^m(\mathbb S^n)} \leq C_m e^{-\delta_m t}. \] \end{thm}
	
	\begin{proof}
		By Proposition~\ref{prop:long-time-existence} and Corollary~\ref{uniform-est}, the solution exists smoothly for all $t\geq0$ and satisfies uniform higher-order estimates. Thus every sequence $t_j\to\infty$ has a subsequence along which $M_{t_j}$ converges smoothly to a closed hypersurface $M_\infty$. Lemma~\ref{lem:improved-pinching} yields that $\mathring A=0$ on $M_\infty$. The radius estimates exclude a totally geodesic limit, and therefore $M_\infty$ is a geodesic sphere of radius in a compact subinterval of $(0,\pi/2)$. As $W_\ell(\Omega_t)$ is preserved, every subsequential limiting sphere has the same radius $r_\infty$, uniquely determined by $W_\ell(B_{r_\infty}) = W_\ell(\Omega_0).$ 
		
		This also implies that the radii $\rho(t)$ of the geodesic spheres $S_{\rho(t)}(q(t))$ from Corollary~\ref{cor:quantitative-spherical-closeness} converge to $r_\infty$ as $t\to\infty$. In fact, this convergence is exponential. Indeed, the exponential $C^2$-closeness gives \[ \left| W_\ell(B_{\rho(t)})-W_\ell(\Omega_t) \right| \leq Ce^{-\delta t}. \] 
		As $W_\ell(\Omega_t) = W_\ell(\Omega_0) = W_\ell(B_{r_\infty})$ and the derivative of $\rho\to W_\ell(B_\rho)$ is bounded away from zero on the compact interval containing the radii $\rho(t)$, we conclude that \[ |\rho(t)-r_\infty| \leq Ce^{-\delta t}. \] Together with the exponential $C^2$-decay of the graph functions, this yields \[ \left| A-\cot r_\infty\,\mathrm{Id} \right| \leq Ce^{-\delta t}. \tag{7.5} \label{eq:weingarten-convergence} \]
		
		Now set $F_\infty := F(\cot r_\infty,\ldots,\cot r_\infty)=n\cot r_\infty$, and $\sigma_{\ell,\infty} := \binom{n}{\ell}\cot^\ell r_\infty.$ Since the curvature vectors remain in a fixed compact subset of $\Gamma_+$, the smoothness of $F$ and $\sigma_\ell$, together with \eqref{eq:weingarten-convergence}, gives \[ F = F_\infty+O(e^{-\delta t}), \qquad \sigma_\ell = \sigma_{\ell,\infty}+O(e^{-\delta t}). \tag{7.6} \label{eq:F-sigma-asymptotics} \] Here and below, the error terms are uniform, and the decay rate may decrease from line to line. Substituting \eqref{eq:F-sigma-asymptotics} into the definition of the global term gives \[ \mu(t) = \frac{F_\infty} {\fint_{M_t}\phi'\,dV_t} + O(e^{-\delta t}), \tag{7.7} \label{eq:mu-asymptotics} \]
		where  we have used the notation $ \fint_{M_t}\phi'\,dV_t := \frac{1}{|M_t|} \int_{M_t}\phi'\,dV_t$. 
		
		It remains to show that every limiting sphere is centered at the fixed origin $\mathcal O$. We follow the argument in \cite[Section~7]{CabezasRivasScheuer2024}, with the following  modifications. Since we do not assume that $\mathcal O \in \Omega_t$, we work at a maximum point of the radial distance. Define \[ R(t):=\max_{M_t}r \] and let $x_t\in M_t$ be a point where this maximum is attained. At $x_t$, the convex hypersurface $M_t$ is supported from the outside by the geodesic sphere $S_{R(t)}(\mathcal O)$, and hence $ \nu(x_t,t)=\partial_r.$ Therefore, using the upper Dini derivative and the flow equation, we obtain \[ D^+R(t) \leq \mu(t)\phi'(x_t,t)-F(x_t,t). \] Combining this inequality with \eqref{eq:F-sigma-asymptotics} and \eqref{eq:mu-asymptotics} yields \[ D^+R(t) \leq F_\infty \left( \frac{\phi'(x_t,t)} {\fint_{M_t}\phi'\,dV_t} -1 \right) + O(e^{-\delta t}). \tag{7.8} \label{eq:max-radius-asymptotics} \] 
		
		As $F_\infty=n\cot r_\infty>0,$ inequality~\eqref{eq:max-radius-asymptotics} has the same form as the asymptotic radial inequality in \cite[Section~7]{CabezasRivasScheuer2024}, up to an exponentially decaying error. The argument there therefore applies without further change to exclude every noncentered geodesic sphere as a subsequential limit. Therefore every subsequential limit is $S_{r_\infty}(\mathcal O)$. 
		
		Taking into account that the family $\{M_t\}_{t\geq0}$ is smoothly precompact and has a unique subsequential limit, the whole flow converges smoothly to $S_{r_\infty}(\mathcal O)$. 
		
		We finally prove that the convergence is exponential, which is a conclusion not esta\-blished in \cite{CabezasRivasScheuer2024}. The argument is a direct maximum-principle argument for the radial graph over the limiting sphere, followed by interpolation. A related strategy, in which exponential decay of a low-order quantity is upgraded by interpolation, is used in \cite[Section~3.4]{MiglioranzaThesis}; for the interpolation step, see also \cite[Lemma~C.2]{SchnurerSmoczyk}.
		
		Set $R:=r_\infty$. Since $M_t$ converges smoothly to $S_R(\mathcal O)$, for all sufficiently large $t$ it can be written as a radial graph \[ M_t = \left\{ \bigl(R+f(\theta,t),\theta\bigr): \theta\in\mathbb S^n \right\}, \] where $f(\cdot,t)\to0$ smoothly. We first control the average of the graph function, which
		corresponds to a variation of the radius of the centered sphere. For
		a sufficiently small function $f$ on $S_R(\mathcal O)$, let $\Omega_f$
		denote the domain bounded by its radial graph and define $\mathcal W(f):=W_\ell(\Omega_f)$, and hence by $W_\ell$-preservation under the flow it holds
		$\mathcal W(f(\cdot,t)) = W_\ell(\Omega_t) = W_\ell(B_R) = \mathcal W(0).
		$ As $\mathcal W$ is smooth in a neighbourhood of the zero graph, using \eqref{eq:first-variation-Wl}
		Taylor's formula gives
		\[
		0
		=
		c_{n,\ell}\sigma_{\ell,\infty}
		\int_{S_R(\mathcal O)}f\,dV_\infty
		+
		O\bigl(\|f\|_{C^2(S_R(\mathcal O))}^2\bigr),
		\]
		from where we reach
		\[
		\bigg|
		\fint_{S_R(O)}f\,dV_\infty
		\bigg|
		\leq
		C\|f\|_{C^2(S_R(O))}^2.
		\tag{7.9}
		\label{eq:mean-graph-control}
		\]
		
		Define
		$f_{\max}(t)
		:=
		\max_{\mathbb S^n}f(\cdot,t)$ and
		$f_{\min}(t)
		:=
		\min_{\mathbb S^n}f(\cdot,t).$ 
		At a spatial maximum or minimum of $f$, its spatial gradient
		vanishes and the radial graph factor is equal to one. Therefore the
		radial evolution equation, \eqref{eq:F-sigma-asymptotics} and \eqref{eq:mu-asymptotics} yield
		\[
		D^+f_{\max}(t)
		\leq
		F_\infty
		\bigg(
		\frac{\cos(R+f_{\max}(t))}
		{\fint_{M_t}\phi'\,dV_t}
		-1
		\bigg)
		+
		Ce^{-\delta t},
		\tag{7.10}
		\label{eq:max-graph-evolution}
		\]
		and
		\[
		D_-f_{\min}(t)
		\geq
		F_\infty
		\bigg(
		\frac{\cos(R+f_{\min}(t))}
		{\fint_{M_t}\phi'\,dV_t}
		-1
		\bigg)
		-
		Ce^{-\delta t},
		\tag{7.11}
		\label{eq:min-graph-evolution}
		\]
		where $D_-$ denotes the lower Dini derivative.
		
		Since $f\to0$ smoothly, expansion of both the integrand and the area
		element of the radial graph yields
		\[
		\fint_{M_t}\phi'\,dV_t
		=
		\cos R
		-
		\sin R
		\fint_{S_R(\mathcal O)}f\,dV_\infty
		+
		O\bigl(\|f\|_{C^1(S_R(\mathcal O))}^2\bigr) =
		\cos R
		+
		O\bigl(\|f\|_{C^2(S_R(\mathcal O))}^2\bigr),
		\tag{7.12}
		\]
		after using \eqref{eq:mean-graph-control}, and enlarging the error term. Taking this into account, as well as $\cos(R+s)
		=
		\cos R-\sin R\,s+O(s^2)$ and $F_\infty\tan R=n$, we can write
		\[
		D^+f_{\max}(t)
		\leq
		-nf_{\max}(t)
		+
		C\|f(\cdot,t)\|_{C^2}^2
		+
		Ce^{-\delta t},
		\tag{7.13}
		\label{eq:max-graph-decay}
		\]
		and
		\[
		D_-f_{\min}(t)
		\geq
		-nf_{\min}(t)
		-
		C\|f(\cdot,t)\|_{C^2}^2
		-
		Ce^{-\delta t}.
		\tag{7.14}
		\label{eq:min-graph-decay}
		\]
		
		Now set
		$\eta(t)
		:=
		\max\bigl\{
		f_{\max}(t),-f_{\min}(t)
		\bigr\}
		=
		\|f(\cdot,t)\|_{C^0}.$
		Since $M_t$ converges smoothly to $S_R(\mathcal O)$ and the graphical
		coordinate changes are uniformly controlled, 
		Corollary~\ref{uniform-est} implies that, for every integer
		$N\geq0$ and all sufficiently large $t$, we have
		\[
		\|f(\cdot,t)\|_{C^N(S_R(\mathcal O))}
		\leq
		C_N.
		\]

		Choose an integer $N>4$. By interpolation, for every $\varepsilon>0$ there exists
		$t_\varepsilon$ such that
		\[
		\begin{aligned}
		\|f(\cdot,t)\|_{C^2}^2
		&\leq
		C
		\|f(\cdot,t)\|_{C^0}^{\,2-4/N}
		\|f(\cdot,t)\|_{C^N}^{\,4/N}\leq
		C\eta(t)^{\,2-4/N} \leq \varepsilon\eta(t)
		\end{aligned}
		\tag{7.15}
		\label{eq:C2-interpolation}
		\]
		for all $t\geq t_\varepsilon$, which follows because
		$
		2-\frac4N>1$
		and $\eta(t)\to0$.

		Choose $\varepsilon<n/2$. From
		\eqref{eq:max-graph-decay} and \eqref{eq:min-graph-decay} we obtain
		\[
		D^+\eta(t)
		\leq
		-\frac n2\eta(t)
		+
		Ce^{-\delta t}
		\]
		for every sufficiently large $t$. The maximum principle now gives
		\[
		\eta(t)
		=
		\|f(\cdot,t)\|_{C^0}
		\leq
		Ce^{-\gamma t} \qquad \text{for every} \, 0<\gamma<
		\min\left\{
		\frac n2,\delta
		\right\}.
		\tag{7.16}
		\label{eq:graph-C0-exponential-decay}
		\]

		It remains to improve the decay to all spatial derivatives. Fix
		$m\geq1$ and choose an integer $N>m$. By interpolation and the
		uniform $C^N$-estimate,
		\[
		\begin{aligned}
		\|f(\cdot,t)\|_{C^m}
		&\leq
		C
		\|f(\cdot,t)\|_{C^0}^{\,1-m/N}
		\|f(\cdot,t)\|_{C^N}^{\,m/N}\leq
		C_m
		e^{-\gamma(1-m/N)t} \leq C_m e^{-\gamma_m t},
		\end{aligned}
		\]
		with $\gamma_m
		:=
		\gamma
		\left(
		1-\frac mN
		\right)>0$.
		Therefore $M_t$ converges exponentially in $C^\infty$ to
		$S_{r_\infty}(\mathcal O)$, as desired.
	\end{proof}
	
	Theorem~\ref{thm:exponential-convergence}, together with Proposition~\ref{prop:long-time-existence}, completes the proof of Theorem~\ref{thm:main}.

    \section*{Declarations}

    \subsection*{Data Availability} This manuscript has no associated data

    \subsection*{Competing Interests} The authors have no competing interests to declare that are relevant to the content of this article.

\end{document}